\documentclass[11pt,reqno,oneside]{amsart}
\usepackage[T1]{fontenc}
\usepackage[english]{babel}
\usepackage{latexsym,amsmath,amsfonts,amscd,amssymb}
\usepackage{graphics}
\usepackage{float}
\usepackage[all]{xy}
\usepackage{amsthm}
\usepackage{color}
\definecolor{cobalt}{RGB}{61,89,171}
\usepackage{makecell}
\usepackage{adjustbox}
\usepackage{longtable}
\usepackage{array, multicol, multirow, makecell}
\usepackage{nicefrac}

\usepackage[hypertexnames=false,
backref=page,
    pdftex,
    pdfpagemode=UseNone,
    breaklinks=true,
    extension=pdf,
    colorlinks=true,
    linkcolor=blue,
    citecolor=blue,
    urlcolor=blue,
]{hyperref}

\newcommand{\frg}{{\mathfrak{g}}}
\newcommand{\frh}{{\mathfrak{h}}}
\newcommand{\fra}{{\mathfrak{a}}}

\newcommand\Real{{\mathfrak R}{\frak e}\,} 
\newcommand\Imag{{\mathfrak I}{\frak m}\,}

\theoremstyle{plain}
\newtheorem{theorem}{Theorem}[section]
\newtheorem{definition}[theorem]{Definition}
\newtheorem{lemma}[theorem]{Lemma}
\newtheorem{proposition}[theorem]{Proposition}
\newtheorem{corollary}[theorem]{Corollary}
\newtheorem{remark}[theorem]{Remark}

\newtheorem{remark-question}[section]{Remark-Question}

\title[]{Nilmanifolds with non-nilpotent complex structures and their pseudo-K\"ahler geometry}

\author{Adela Latorre}
\address[A. Latorre]{Departamento de Matem\'atica Aplicada,
Universidad Polit\'ecnica de Madrid,
Avda. Juan de Herrera 4,
28040 Madrid, Spain}
\email{adela.latorre@upm.es}

\author{Luis Ugarte}
\address[L. Ugarte]{Departamento de Matem\'aticas\,-\,I.U.M.A.\\
Universidad de Zaragoza\\
Campus Plaza San Francisco\\
50009 Zaragoza, Spain}
\email{ugarte@unizar.es}


\begin{document}

\begin{abstract}
We classify nilpotent Lie algebras with complex structures of weakly non-nilpotent type in real dimension eight, which is the lowest dimension where they arise. Our study, together with previous results on strongly non-nilpotent structures, completes the classification of 8-dimensional nilpotent Lie algebras admitting complex structures of non-nilpotent  type. As an application, we identify those that support a pseudo-K\"ahler metric, thus providing new counterexamples to a previous conjecture and an infinite family of (Ricci-flat) non-flat neutral Calabi-Yau structures. Moreover, we arrive at the topological restriction $b_1(X)\geq 3$ for every pseudo-K\"ahler nilmanifold~$X$ with an invariant complex structure, up to complex dimension four.
\end{abstract}

\maketitle

%
%
\section{Introduction}\label{}

\noindent This paper provides another step towards a better understanding of  nilmanifolds endowed with {\it invariant} complex structures $J$. It is well known that these spaces constitute an important class to explore many aspects of the geometry of compact complex non-K\"ahler manifolds. Concretely, our first goal here is to classify the nilpotent Lie algebras of real dimension~8 that support a complex structure $J$ of a special type called {\it weakly non-nilpotent}, {\it WnN} for short (see Definition~\ref{tipos_J} and Remark~\ref{classes-of-J}). 
This classification, together with the results on strongly non-nilpotent 
complex structures obtained in \cite{LUV-SnN-2}, gives rise to a complete classification of the 8-dimensional nilpotent Lie algebras admitting non-nilpotent $J$'s. 
Therefore, our results reduce the problem of classifying nilpotent Lie algebras with complex structures in the aforementioned dimension to the study of nilpotent complex structures.

Let us recall that the classification of complex structures on nilpotent Lie algebras of real dimensions $4$ and $6$ is well understood. In dimension $4$ there are only two complex structures defined on two different nilpotent Lie algebras that give rise to the complex torus and the Kodaira-Thurston manifold. In six dimensions, only $18$ of the $34$ non-isomorphic nilpotent Lie algebras admit a complex structure \cite{Salamon}, and the classification of complex structures up to equivalence on each
of these algebras is studied in \cite{COUV} (see also \cite{ABD,UV-SnN}). Every complex structure $J$ in real dimension 4 is nilpotent, whereas in dimension 6 it is either nilpotent or strongly non-nilpotent; there do not exist WnN complex structures in dimension $\leq 6$.

In real dimension 8 there are WnN complex structures~\cite{LUV-SnN}. These structures can be viewed as a subclass of \emph{quasi-nilpotent} complex structures $J$, which are defined by the condition 
$\mathcal Z(\frg) \cap J(\mathcal Z(\frg)) \not= \{0\}$,
where $\mathcal Z(\frg)$ denotes 
the center of the nilpotent Lie algebra~$\frg$. 
Inside this class, one can find well-known nilpotent complex structures, such as abelian or complex-parallelizable, but also non-nilpotent ones. 
Since for any quasi-nilpotent complex structure $J$ there is a $J$-invariant subspace in the center $\mathcal Z(\frg)$, by~\cite[Section 2]{LUV-SnN}, 
every $(\frg,J)$ 
can be constructed
as a certain extension of a lower dimensional nilpotent Lie algebra endowed with a complex structure.

In Section~\ref{sec:preliminaries}, after recalling the definitions and main properties of the different types of complex structures on nilpotent Lie algebras, we focus our attention on quasi-nilpotent structures $J$. We show in Proposition~\ref{f-iso-sucesion-a-Prop} that the equivalence relation for two such $J$'s is inherited by the induced complex structures $\tilde J$'s on the corresponding quotient algebras.

The previous result is applied to WnN structures in eight dimensions. 
Firstly, we construct in Section~\ref{subsec:construction-WnN} all the WnN complex structures by means of their complex structure equations. 
Hence, after a careful reduction of the equations (see Propositions~\ref{reduc-nu-0} and~\ref{reduc-nu-1}), we arrive at the classification of WnN complex structures up to equivalence in Section~\ref{subsec:classification-complex}. 
The final classification result corresponds to Theorem~\ref{main-theorem}. 

In Section~\ref{clasi-real} we classify the 8-dimensional nilpotent Lie algebras that admit WnN complex structures. It turns out that there are exactly 10 non-isomorphic algebras, four of them having an infinite number of (non-equivalent) complex structures. Theorem~\ref{main-theorem-2} provides the classification of algebras, whereas Table~1 describes the moduli spaces of WnN complex structures on each one of the ten nilpotent Lie algebras. 
Furthermore, the algebras are all rational, so they give rise to new complex nilmanifolds $X$ with $\dim_{\mathbb C}X=4$. 


Our second goal in this paper is to  
apply the classification results on WnN complex structures $J$ to the study of existence of pseudo-K\"ahler and neutral Calabi-Yau metrics on the corresponding complex nilmanifolds $X=(M,J)$. Note that  pseudo-K\"ahler Lie groups can be used in the construction of Sasaki-Einstein solvmanifolds of indefinite signature \cite{CRD}. 
In real dimension 4, the Lie algebras carrying a pseudo-K\"ahler structure are classified in \cite{Ovando}. For nilpotent Lie algebras of real dimension 6, the classification is given in \cite{CFU}. However, a classification in eight dimensions is not known. 

It was conjectured in \cite{CFU} that $J$ must be nilpotent in the presence of a pseudo-K\"ahler structure. Recently, a counterexample was given by the authors in eight dimensions \cite{LU-PuresAppl}, 
which in addition admits (Ricci-flat) non-flat neutral Calabi-Yau metrics. Moreover, it
was used to prove that the pseudo-K\"ahler, neutral K\"ahler and neutral Calabi-Yau properties are not stable under small deformations of the complex structure, in any complex dimension $n>2$. 
In Section~\ref{clasi-pseudoK} we identify all the non-nilpotent complex structures admitting a pseudo-K\"ahler metric (see Theorem~\ref{class-pK-nN-dim8}). We show that the counterexample in eight dimensions found in \cite{LU-PuresAppl} 
is unique in the class of strongly non-nilpotent complex structures. We also find an infinite family of new counterexamples in the class of WnN complex structures.

In Theorem~\ref{new-neutral-CY} we prove that 
the families of pseudo-K\"ahler metrics obtained in Theorem~\ref{class-pK-nN-dim8} 
provide new neutral Calabi-Yau metrics in eight dimensions that are (Ricci-flat) non-flat. 
We recall here that many examples of neutral Calabi-Yau metrics in the literature come from hypersymplectic structures \cite{Hit}. 
Indeed, several constructions of such structures on nilpotent Lie algebras have been obtained, as the 2-step examples (of Kodaira type) in \cite{FPPS} or the 3-step examples given in \cite{AD} (see also \cite{Andrada}). The first 4-step nilpotent Lie algebra with a hypersymplectic structure can be found in \cite{BGL}. New 2-step nilpotent hypersymplectic Lie algebras whose underlying complex structure is not abelian appear in \cite{ContiGil}.
Note that the pseudo-K\"ahler Lie algebras whose underlying $J$ is abelian can be inductively described by a certain method of double extensions \cite{BajoS}. 
All the known hypersymplectic nilpotent Lie algebras are, in particular, neutral Calabi-Yau (hence, also pseudo-K\"ahler). It is worthy to mention that
their underlying complex structures $J$ are of nilpotent type. 

In Proposition~\ref{non-existence-hol-symplectic} we prove that the neutral Calabi-Yau structures constructed in Theorem~\ref{new-neutral-CY} do not come from a hypersymplectic structure. Moreover, we show that non-nilpotent complex structures in real dimension 8 do not admit any complex symplectic form, thus extending the non-existence result  for strongly non-nilpotent complex structures given in \cite[Proposition 5.8]{BFLM}. 

It is well known that the existence of a (positive definite) K\"ahler metric on a compact complex manifold imposes strong restrictions on the topology of the manifold. For nilmanifolds, the only (positive definite) K\"ahler spaces are tori \cite{Hasegawa}. However, in the pseudo-K\"ahler setting no topological obstruction is known, to our knowledge, on the compact complex manifold, apart from those coming from the existence of a symplectic form. 
In Theorem~\ref{restrictions-pK} we arrive at the topological restriction $b_1(X)\geq 3$ for every pseudo-K\"ahler nilmanifold $X$ with an invariant complex structure, up to complex dimension~4. This result suggests the existence of a possibly more general topological obstruction, at least for nilmanifolds. 

%
%

%
%
\section{Complex structures on nilpotent Lie algebras}\label{sec:preliminaries}

\noindent The aim of this section is to study invariant complex structures on nilmanifolds. As these complex structures are determined by complex structures defined on the nilpotent Lie algebra associated to the nilmanifold, we will mainly work at the Lie algebra level.

\medskip 
Let $\frg$ be a real Lie algebra. The \emph{ascending central series} 
$\{\frg_k\}_{k\geq 0}$ of $\frg$ is defined as
\begin{equation}\label{ascending-central-series}
	\frg_0=\{0\} \quad \text{ and } \quad
	\frg_k=\{X\in\frg \mid [X,\frg]\subseteq \frg_{k-1}\}, \text{ for } k\geq 1.
\end{equation}
Observe that $\frg_1=\mathcal Z(\frg)$ is, precisely, the \emph{center} of~$\frg$. 
If there is an integer~$k\geq 1$ such that~$\frg_k=\frg$, the Lie algebra $\frg$ is called \emph{nilpotent}. The smallest integer $s$
for which $\frg_s=\frg$ is called the \emph{nilpotency step} of~$\frg$. If $\mathfrak g$ is a nilpotent Lie algebra (NLA for short), the corresponding Lie group $G$ is also  
called nilpotent.

We recall that a \emph{nilmanifold} $\Gamma\backslash G$ is a compact quotient of a connected, simply connected, nilpotent Lie group $G$ by a lattice $\Gamma\leq G$. Any basis of the Lie algebra $\mathfrak g$ of $G$ generates a basis of left-invariant vector fields on $G$ that can be transferred to the nilmanifold $\Gamma\backslash G$. Moreover, given $G$, the existence of a lattice $\Gamma$ making $\Gamma\backslash G$ a compact quotient is guaranteed by the existence of a basis of $\mathfrak g$ where the structure constants are rational numbers~\cite{Mal}. Thus, in this sense one can identify a nilmanifold $\Gamma\backslash G$ with the \emph{structure equations} of its Lie algebra, namely,
\begin{equation}\label{def-ecus-estructura}
de^k=\sum_{1\leq i<j\leq m} c_{ij}^k\,e^i\wedge e^j, \quad 1\leq k\leq m=\dim\frg,
\end{equation}
where $\{e^k\}_k$ is a basis of the dual space $\mathfrak g^*$ of $\mathfrak g$. To describe a Lie algebra defined by \eqref{def-ecus-estructura} 
we will write $\frg=(\sum_{i<j} c_{ij}^1\!\cdot\! i j,\ldots, \sum_{i<j} c_{ij}^m\!\cdot\! i j)$. For instance, $\mathfrak h=(0,0,0,2\!\cdot\!13)$ means that $\mathfrak h$ is a $4$-dimensional Lie algebra with structure equations $de^1=de^2=de^3=0$, $de^4=2\, e^1\wedge e^3$. Remember that the Lie bracket of any Lie algebra $\mathfrak g$ can be recovered from~\eqref{def-ecus-estructura} using the well-known formula
\begin{equation}\label{rel-diferencial-corchete}
d\alpha(X,Y)=-\alpha([X,Y]), \quad 
	\forall\,\alpha\in\frg^*, \ \forall\,X,Y\in\frg.
\end{equation}

If $G$ is a complex Lie group, then the nilmanifold $\Gamma\backslash G$ is a complex(-parallelizable) manifold. However, one can also construct complex nilmanifolds using even-dimensional real Lie groups admitting complex structures. In fact, in this paper we are interested in finding \emph{invariant} complex structures $J$ on real nilmanifolds, which  are those $J$'s determined by complex structures directly constructed on the Lie algebra $\mathfrak g$ of $G$. 

Recall that a \emph{complex structure} $J$ on a $2n$-dimensional Lie algebra $\frg$ is an endomorphism 
$J\colon\frg\to\frg$ satisfying $J^2=-\textrm{Id}$ and the integrability condition
$$
N_J(X,Y):=[X,Y]+J[JX,Y]+J[X,JY]-[JX,JY]=0,
$$
for all $ X,Y\in\frg$.

\subsection{Different types of complex structures on NLAs}\label{subsec:different-types}

Let $\mathfrak g$ be a nilpotent Lie algebra with a complex structure $J$. The terms~$\frg_{k}$ in the ascending central series~\eqref{ascending-central-series} might not be invariant under~$J$, as observed 
in~\cite{CFGU-dolbeault}. For this reason, a new series is introduced considering both the complex structure~$J$ and the structure of the Lie algebra~$\frg$. It is called the \emph{ascending $J$-compatible series} 
$\{\fra_{k}(J)\}_{k}$ of~$\frg$, and it is defined as follows:
$$
\begin{cases}
\fra_0(J)=\{0\}, \text{ and } \\
\fra_k(J)=\{X\in\frg \mid [X,\frg]\subseteq \fra_{k-1}(J)\ {\rm and\ } [JX,\frg]\subseteq \fra_{k-1}(J)\}, \text{ for } k\geq 1.
\end{cases}
$$
Observe that each $\fra_k(J)\subseteq\frg_{k}$ is a $J$-invariant ideal of $\frg$, so it has even dimension. Moreover, 
$\fra_1(J)$ is the largest $J$-invariant subspace contained in the center $\frg_1$ of $\mathfrak g$. 

The behaviour of the series~$\{\fra_{k}(J)\}_{k}$ determines different types of complex structures on NLAs.

\begin{definition}\label{tipos_J}\cite{CFGU-dolbeault, LUV-SnN}
{\rm
A complex structure $J$ on a nilpotent Lie algebra $\frg$ is called
\begin{itemize}
\item[i)] \emph{strongly non-nilpotent}, or \emph{SnN} for short, if $\fra_1(J)=\{0\}$;

\smallskip
\item[ii)] \emph{quasi-nilpotent}, if $\fra_1(J)\neq\{0\}$; moreover, $J$ is called
 \begin{itemize}
 \item[a)] \emph{nilpotent}, if there exists an integer~$t>0$ such that~$\fra_t(J)=\frg$,
 \item[b)] \emph{weakly non-nilpotent}, or \emph{WnN} for short, if there is an integer~$t>0$ satisfying~$\fra_t(J)=\fra_{t+l}(J)$, for every~$l\geq 1$,
           and~$\fra_t(J)\neq\frg$.
 \end{itemize}
\end{itemize}
}
\end{definition}

\begin{remark}\label{classes-of-J}
Let ${\mathcal C}(\frg)$ be the set of complex structures $J$ on a nilpotent Lie algebra $\frg$. From the definitions, one has
${\mathcal C}(\frg)={\mathcal{QN}}\!(\frg) \mathop{\dot{\cup}} {\mathcal{S{\it n }N}}\!(\frg)= {\mathcal N}\!(\frg) \mathop{\dot{\cup}} {\mathcal{{\it n }N}}\!(\frg) $, and  
$${\mathcal{QN}}\!(\frg) = {\mathcal N}\!(\frg) \mathop{\dot{\cup}} {\mathcal{W{\it n }N}}\!(\frg), \quad {\mathcal{{\it n }N}}\!(\frg)={\mathcal{W{\it n }N}}\!(\frg) \mathop{\dot{\cup}} {\mathcal{S{\it n }N}}\!(\frg),$$
where ${\mathcal{QN}}\!(\frg)$, ${\mathcal N}\!(\frg)$, ${\mathcal{W{\it n }N}}\!(\frg)$, ${\mathcal{{\it n }N}}\!(\frg)$ and ${\mathcal{S{\it n }N}}\!(\frg)$ denote, respectively, the spaces of quasi-nilpotent, nilpotent, WnN, non-nilpotent and SnN complex structures on $\frg$. 
\end{remark}

We recall that every complex structure $J$ on a $4$-dimensional NLA $\frg$ is of nilpotent type. 
In dimension $6$, a complex structure $J$ is either nilpotent or strongly non-nilpotent. 
The first examples of weakly non-nilpotent complex structures appear in real dimension~$8$ (see, for instance,~\cite[Example~2.3]{LUV-SnN}).

Let $J$ be a quasi-nilpotent complex structure on a $2n$-dimensional NLA $\frg$. Then, there exists $t\in\mathbb N$ such that
$$
\{0\}=\fra_0(J)\subsetneq\fra_1(J)\subsetneq\ldots\subsetneq\fra_{t-1}(J)\subsetneq\fra_{t}(J)=\fra_{t+l}(J), \ \text{ for } l\geq 1.
$$
For any $q\geq 1$ such that $\fra_q(J)\not=\frg$, we consider 
the quotient Lie algebra
\begin{equation}\label{induced-gq}
\tilde\frg_q=\frg\slash\fra_q(J),
\end{equation}
with the induced complex structure
\begin{equation}\label{induced-Jq}
\tilde J_q(\tilde X) := \widetilde{JX}, \quad \text{ for all }\tilde X\in\tilde\frg_q,
\end{equation}
where $\tilde X$ and $\widetilde{JX}$ denote, respectively, the classes of $X$ and $JX$ in $\tilde\frg_q$. The terms of the compatible series of $(\tilde\frg_q,\tilde J_q)$ and $(\frg,J)$
are related as follows~\cite[Lemma~3]{CFGU-proceeding}:

\begin{lemma}\label{lema-cociente}\cite{CFGU-proceeding}
Let $J$ be a quasi-nilpotent complex structure on $\frg$. Consider a non-trivial term $\fra_q(J)$ of the ascending $J$-compatible series of $\frg$. 
Then, the ascending $\tilde J_q$-compatible series $\{\fra_l(\tilde J_q)\}_l$ of the quotient Lie algebra $\tilde\frg_q$ with induced complex structure~$\tilde J_q$ satisfies
$$\fra_l(\tilde J_q)\cong\fra_{q+l}(J)/\fra_q(J), \text{ for all }l\geq 0.$$
\end{lemma}

Suppose now that $J$ is weakly non-nilpotent.
Then, one can choose $q$ in Lemma~\ref{lema-cociente} to be the smallest integer $t$ where the ascending $J$-compatible series $\{\fra_k(J)\}_k$ of $\frg$ stabilizes, namely, $\fra_t(J)=\fra_{t+l}(J)$ for every $l\geq 1$. Since 
$$
\fra_l(\tilde J_t) \cong\fra_{t+l}(J)/\fra_t(J)=\{0\}, \text{ for every }l\geq 0,
$$
the complex structure $\tilde J_t$ defined by~\eqref{induced-Jq} on the quotient Lie algebra $\tilde\frg_t=\frg\slash\fra_t(J)$ is of strongly non-nilpotent type. 
Hence, $2n$-dimensional NLAs with weakly non-nilpotent complex structures are closely related to lower-dimensional NLAs admitting strongly non-nilpotent complex structures. As six is the lowest dimension where SnN structures appear, one immediately has the following result in real dimension $2n=8$:

\begin{lemma}\label{lema1}
	Let $(\frg,J)$ be an 8-dimensional NLA with a weakly non-nilpotent complex structure. Let~$t$ be the smallest integer for which the ascending $J$-compatible series
$\{\fra_{k}(J)\}_{k}$ of~$\frg$ satisfies $\{0\}\neq\fra_t(J)=\fra_{t+l}(J)\neq\frg$, for every $l\geq 1$. Then, 
	$(\tilde\frg_t, \tilde J_t)$
	is a $6$-dimensional NLA with an SnN complex structure. In particular, one has $\dim\fra_t(J)=2$ and thus $\fra_t(J)=\fra_1(J)$.
\end{lemma}

\subsection{Equivalence of quasi-nilpotent complex structures}\label{subsec:equivalence}

To classify a geometric structure, it is natural to introduce a notion of equivalence. In the case of complex structures on Lie algebras, one has the following one:

\begin{definition}\label{def-equivalence}
Two complex structures $J$ and $J'$ defined, respectively, on the Lie algebras $\frg$ and $\frg'$
are said to be 
\emph{equivalent} if there exists an isomorphism of Lie algebras $f:\frg\to\frg'$ satisfying $f\circ J=J'\circ f$.
\end{definition}

We next show that the notion of equivalence 
behaves appropriately with respect to the ascending $J$-compatible series.

\begin{lemma}\label{f-iso-sucesion-a-Lema}
Let $(\frg,J)$ and $(\frg',J')$ be two NLAs with complex structures. 
If there is an equivalence between $J$ and $J'$ due to $f:\frg\to\frg'$, then $f\big(\fra_k(J)\big)=\fra_k(J')$, for every $k\geq 0$.
\end{lemma}

\begin{proof}
First, we will prove by induction that $f\big(\fra_k(J)\big) \subseteq \fra_k(J')$ for every $k\geq 0$. Since this is clear for $k=0$, let us suppose that it holds for $k-1$. Then, for any homomorphism of Lie algebras $f:\frg\to\frg'$ such that $f\big(\fra_{k-1}(J)\big) \subseteq \fra_{k-1}(J')$, we have
\begin{eqnarray*}
f\big(\fra_k(J)\big) \!&\!\!=\!\!&\! \big\{ f(X)\in\frg' \mid [X,\frg]\subseteq \fra_{k-1}(J) \mbox{ and } [JX,\frg]\subseteq \fra_{k-1}(J) \big\} \\
\!&\!\!\subseteq\!\!&\! \big\{ f(X)\in\frg' \mid f\big([X,\frg]\big)\subseteq f(\fra_{k-1}(J)) \mbox{ and }  f\big([JX,\frg]\big)\subseteq f(\fra_{k-1}(J)) \big\} \\
\!&\!\!\subseteq\!\!&\! \big\{ f(X)\in\frg' \mid [f(X),f(\frg)]\subseteq \fra_{k-1}(J') \mbox{ and }  [f(JX),f(\frg)]\subseteq \fra_{k-1}(J') \big\}.
\end{eqnarray*}
Now, as $f$ is an isomorphism of Lie algebras satisfying $f\circ J=J'\circ f$, one gets that $f(\frg)=\frg'$ and thus
\begin{eqnarray*}
f\big(\fra_k(J)\big) \!&\!\!\subseteq\!\!&\! 
\big\{ f(X)\!\in\frg' \mid [f(X),\frg']\subseteq \fra_{k-1}(J') \mbox{ and }  [J'\big(f(X)\big),\frg']\subseteq \fra_{k-1}(J') \big\} = \fra_k(J').
\end{eqnarray*}
We now observe that the map $f^{-1}:\frg'\to\frg$ is also an isomorphism of Lie algebras that additionally satisfies $f^{-1}\circ J'=J\circ f^{-1}$. Thus, 
$f^{-1}\big(\fra_k(J')\big) \subseteq \fra_k(J)$ for every $k\geq 0$. 
Applying $f$, we get $\fra_k(J') \subseteq f\big(\fra_k(J)\big)$ and the statement of the lemma is proved.
\end{proof}

A direct consequence of the previous result is that two equivalent complex structures must be of the same type, according to Definition~\ref{tipos_J}. 
The next result shows that the notion of equivalence also
behaves appropriately with quotients, namely, 
it induces an equivalence between the complex structures~\eqref{induced-Jq} defined on the corresponding quotient Lie algebras~\eqref{induced-gq}.

\begin{proposition}\label{f-iso-sucesion-a-Prop}
Let $(\frg,J)$ and $(\frg',J')$ be two NLAs endowed with quasi-nilpotent complex structures. Fix some $k\geq 1$ such that $\fra_k(J)\not=\frg$, 
and consider the quotient Lie algebras $\tilde\frg_k=\frg\slash\fra_k(J)$ 
and $\tilde\frg'_k=\frg'\slash\fra_k(J')$ with their respective induced complex structures $\tilde J_k$ and $\tilde J'_k$. 
If there is an isomorphism  of Lie algebras $f:\frg\to\frg'$ such that $f\circ J=J'\circ f$, then the map $\tilde{f}_k:\tilde\frg_k\to\tilde\frg'_k$ given by
$$\tilde{f}_k(\tilde X):=\widetilde{f(X)},$$
for every $\tilde X\in \tilde\frg_k$, defines an isomorphism of Lie algebras satisfying $\tilde{f}_k\circ \tilde J_k=\tilde J'_k\circ \tilde{f}_k$.\end{proposition}

\begin{proof}
Let us denote by $\pi_k:\mathfrak g\to\tilde{\mathfrak g}_k$ and $\pi'_k:\mathfrak g'\to\tilde{\mathfrak g}'_k$ the natural projections. Observe that the map $\tilde{f}_k:\tilde\frg_k\to\tilde\frg'_k$ is well-defined by Lemma~\ref{f-iso-sucesion-a-Lema}, and the following diagram 

\centerline{\xymatrix@W=5mm{
\frg\ar[r]^f\ar[d]_{\pi_k} & \,\frg'\ar[d]^{\pi'_k} \\
\tilde\frg_k\ar[r]^{\tilde{f}_k} & \tilde\frg'_k
}}

\noindent is commutative.

Let us see first that $\tilde f_k$ is an homomorphism of Lie algebras. Indeed, note that for any $\tilde X,\tilde Y\in \tilde\frg_k$, one can find $X,Y\in\frg$ such that $\tilde X=\pi_k(X)$ and $\tilde Y=\pi_k(Y)$. Then, as $\pi_k, \pi'_k$ and $f$ are homomorphisms of Lie algebras, we have
\begin{equation*}
\begin{split}
\tilde f_k\big([\tilde X, \tilde Y]_{\tilde\frg_k}\big) &= \tilde f_k\big([\pi_k(X), \pi_k(Y)]_{\tilde\frg_k}\big)
	= \tilde f_k\big(\pi_k([X, Y]_{\frg})\big) = \pi'_k\big(f([X,Y]_{\frg})\big)\\
&= \pi'_k\big([f(X),f(Y)]_{\frg'}\big) = \big[(\pi'_k\circ f)(X),(\pi'_k\circ f)(Y)\big]_{\tilde\frg'_k} \\
&= \big[(\tilde f_k\circ \pi_k)(X),(\tilde f_k\circ \pi_k)(Y)\big]_{\tilde\frg'_k} 
	= \big[\tilde f_k(\tilde X),\tilde f_k(\tilde Y)\big]_{\tilde\frg'_k}.
\end{split}
\end{equation*}
Moreover, $\tilde f_k$ is surjective. In fact, for each $\tilde X'\in\tilde \frg'_k$, take $X'\in\frg'$ such that $\pi'_k(X')=\tilde X'$. Since $f$ is an isomorphism, $X'=f(X)$ for some $X\in\frg$. Consequently, we get
$$\tilde X'=\pi'_k(X')=\pi'_k\big( f(X)\big)=\tilde f_k\big(\pi_k(X)\big),$$
where $\pi_k(X)\in\tilde\frg_k$. To conclude that $\tilde f_k$ is an isomorphism, it suffices to observe that $\dim\tilde\frg_k 
= \dim\tilde\frg'_k$, due to Lemma~\ref{f-iso-sucesion-a-Lema}.

Finally, to prove the relation $\tilde{f}_k\circ \tilde J_k=\tilde J'_k\circ \tilde{f}_k$ 
we will use the commutative diagram above, together with the equalities
$$
\pi_k\circ J=\tilde J_k\circ \pi_k, \qquad
\pi'_k\circ J'=\tilde J'_k\circ \pi'_k,
$$
that come directly from~\eqref{induced-Jq}. 
In particular, since $f\circ J=J'\circ f$, we necessarily have $\pi'_k\circ (f\circ J)=\pi'_k\circ (J'\circ f)$. The 
left-hand-side of this equality can be rewritten as 
$$\pi'_k\circ (f\circ J)=(\pi'_k\circ f)\circ J=(\tilde f_k\circ\pi_k)\circ J= \tilde f_k\circ(\pi_k\circ J)
	=\tilde f_k\circ(\tilde J_k\circ \pi_k),$$
whereas the right-hand-side is
$$\pi'_k\circ (J'\circ f)=(\pi'_k\circ J')\circ f=(\tilde J'_k\circ\pi'_k)\circ f= \tilde J'_k\circ(\pi'_k\circ f)
	=\tilde J'_k\circ(\tilde f_k\circ \pi_k).$$
This implies the equality $\tilde{f}_k\circ \tilde J_k=\tilde J'_k\circ \tilde{f}_k$ on the quotient Lie algebra $\tilde\frg_k$, as desired.
\end{proof}

Combining Lemma~\ref{lema1} and Proposition~\ref{f-iso-sucesion-a-Prop}, one has the following: 

\begin{corollary}\label{equiv-WnN-8D-inicio}
Let $(\frg,J)$ and $(\frg',J')$ be two $8$-dimensional NLAs with weakly non-nilpotent complex structures.
If $J$ and $J'$ are equivalent, then the strongly non-nilpotent complex structures $\tilde J_1$ and $\tilde J'_1$ induced, respectively, on the $6$-dimensional Lie algebras $\tilde\frg_1=\frg\slash\fra_1(J)$ and $\tilde\frg'_1=\frg'\slash\fra_1(J')$ are equivalent.
\end{corollary}

\section{Classification of weakly non-nilpotent complex structures\\ 
in real dimension 8}\label{clasi-compleja}

\noindent In this section we focus on real dimension 8, which is the lowest dimension where WnN complex structures arise. The main result is a classification of such structures up to equivalence, in terms of their complex structure equations.

Recall that any almost complex structure $J$ on $\frg$ can be equivalently defined on the complexified dual space $\frg^*_{\mathbb C}=\frg^*\otimes\mathbb C$ by the choice of a direct sum decomposition 
$$\frg^*_{\mathbb C}=\frg_J^{1,0}\oplus\frg_J^{0,1},$$
with $\frg_J^{0,1}=\overline{\frg_J^{1,0}}$. Indeed, $J:\frg^*_{\mathbb C}\to\frg^*_{\mathbb C}$ is defined by setting   
$\frg_J^{1,0}$ as its $i$-eigenspace.  
In greater detail, if $\{\omega^k\}_{k=1}^n$ is any basis of $\frg^{1,0}_J$, then $J$ is defined by taking the real $J$-adapted basis $\{e^k, Je^k\}_{k=1}^n$ of $\mathfrak g^*$ such that $e^k-i\,Je^k:=\omega^k$, for $1\leq k\leq n$.

Observe that the bigraduation on $\mathfrak g^*_{\mathbb C}$ moves to the complexified exterior algebra
$$\Big(\!\bigwedge\nolimits^k\frg^*\!\Big)_{\!\mathbb C}
	\cong\bigwedge\nolimits^k \frg^*_{\mathbb C} 
	=\bigwedge\nolimits^k\big(\frg_J^{1,0}\oplus\frg_J^{0,1} \big)
	= \bigoplus_{p+q=k}\bigwedge\nolimits^{p,q}_J\frg^*,\quad 1\leq k\leq 2n,$$
where $\bigwedge\nolimits^{p,q}_J\frg^*
=\bigwedge^p\big(\frg_J^{1,0}\big)\otimes\bigwedge^q\big(\frg_J^{0,1}\big)$. It is well-known that the integrability condition for $J$ is equivalent to the (extended) 
differential $d$ of~$\frg$ to satisfy
\begin{equation}\label{Nijenhuis-nulo}
d\big(\frg_J^{1,0}\big)\subseteq \bigwedge\nolimits^{2,0}_J\frg^*\bigoplus\bigwedge\nolimits^{1,1}_J\frg^*.
\end{equation}
We will refer to the differentials of any basis $\{\omega^k\}_{k=1}^n$ of $\frg^{1,0}_J$ as the \emph{complex structure equations} of $(\frg,J)$, as those equations completely determine both the NLA $\frg$ and the complex structure $J$. 

The following theorem provides a classification of WnN structures in eight dimensions, up to equivalence.

\begin{theorem}\label{main-theorem}
Let $\frg$ be an $8$-dimensional NLA with a weakly non-nilpotent complex structure $J$. 
Then, there is a basis of $(1,0)$-forms $\{\omega^k\}_{k=1}^4$ in terms of which the complex structure equations of $(\frg,J)$ are 
\begin{equation}\label{structure-eq-WnN}
\left\{
\begin{split}
d\omega^1 &= 0,\\[-4pt]
d\omega^2 &= \omega^{13} + \omega^{1\bar3},\\[-4pt]
d\omega^3 &= i\,\varepsilon\,\omega^{1\bar1} + i\,\delta\,\omega^{1\bar2} 
	- i\,\delta\,\omega^{2\bar1},\\[-4pt]
d\omega^4 &= a\,\omega^{12}+B\,\omega^{1\bar 1} 
	+\nu\,\big(\omega^{23}+2\,\delta\,\varepsilon\,\omega^{1\bar 3}+ \omega^{2\bar 3}\big),
\end{split}
\right.
\end{equation}
where the tuple $(\varepsilon,\delta,\nu,a,B)$ satisfies that $\varepsilon\in\{0,1\}$, $\delta=\pm 1$ and $(\nu,a,B)$ takes one of the following values:
\begin{itemize}
\item $\nu=a=B=0$;
\item $\nu=a=0$, $B=1$;
\item $\nu\in\{0,1\}$, $a=1-\nu$, and $B\in\mathbb R^{\geq 0}$; moreover, $B\in\{0,1\}$ when $\varepsilon=0$;
\item $\nu=1$, $a\in\mathbb R^{> 0}$ and $B\in\mathbb C$; 
furthermore, when $\varepsilon=0$ one has $a=1$ and either $B\in\mathbb R^{\geq 0}$ or  $\Imag B > 0$.
\end{itemize}
Furthermore, different values of $(\varepsilon,\delta,\nu,a,B)$ provide non-equivalent complex structures.
\end{theorem}

The proof of this result consists of three steps, each contained in one of the following subsections. In Section~\ref{subsec:construction-WnN}, we make use of Lemma~\ref{lema1} to find the complex structure equations of every $8$-dimensional NLA $\frg$ with a weakly non-nilpotent complex structure~$J$. Although this initial set of equations provides every possible pair $(\frg,J)$ in terms of some parameters, different choices of these parameters might give rise to equivalent complex structures. For this reason, in Section~\ref{subsec:reduction} we reduce the possible values of the parameters. Finally, in Section~\ref{subsec:classification-complex} we prove that our reduction is optimal, in the sense that it provides a complete classification of WnN  structures on $8$-dimensional NLAs up to equivalence.

\subsection{Construction of WnN from SnN}\label{subsec:construction-WnN}

The aim of this section is to parametrize all the 8-dimensional NLAs with weakly non-nilpotent complex structures.

Let $(\mathfrak g,J)$ be an 8-dimensional NLA with a WnN complex structure.
As a consequence of Lemma~\ref{lema1}, we observe that
$\mathfrak g$ can be seen as a central extension of Lie algebras. Indeed, one has the following short exact sequence

\medskip
\centerline{\xymatrix{
0\ar[r] & \fra_1(J)\ar@{^(->}[r]^{\ \ \iota}& \frg\ar[r]^{\pi_1} & \tilde\frg_1\ar[r] & 0,
}}

\noindent
where $\iota$ denotes the inclusion and $\pi_1$ is the natural projection. We also recall that the subspace $\mathfrak a_1(J)\subseteq\mathcal Z(\mathfrak g)$ is $2$-dimensional and $J$-invariant, whereas $\tilde\frg_1$ has a complex structure $\tilde J_1$ induced by $J$.
This observation allows us to provide a first description of the pairs $(\frg,J)$ in terms of their complex structure equations. 

\begin{proposition}\label{prop-inicial}
Let $J$ be a WnN complex structure on an $8$-dimensional NLA~$\frg$. Then, there is a basis of $(1,0)$-forms $\{\eta^k\}_{k=1}^4$ in terms of which the complex structure equations of~$(\frg,J)$ are
\begin{equation}\label{ecus-1}
	\begin{cases}
	d\eta^1=0,\\
	d\eta^2=\eta^{13}+\eta^{1\bar3},\\
	d\eta^3=i\,\varepsilon\,\eta^{1\bar1}+i\,\delta\,\eta^{1\bar2}-i\,\delta\,\eta^{2\bar1},\\
	d\eta^4=A_\eta\,\eta^{12}+B_\eta\,\eta^{1\bar1}
		+\nu\,\big(\eta^{23}+2\,\delta\,\varepsilon\,\eta^{1\bar3}+\eta^{2\bar3}\big),
	\end{cases}
\end{equation}
where $\varepsilon,\nu\in\{0,1\}$, $\delta=\pm1$, and $A_\eta,B_\eta\in\mathbb C$.
\end{proposition}

\begin{proof}
Since $\mathfrak a_1(J)$ is a $2$-dimensional $J$-invariant ideal of $\mathfrak g$, we can take a basis $\{Z_k,\bar Z_k\}_{k=1}^4$ for the complexified algebra $\frg_{\mathbb C}$, with each $Z_k$ of bidegree $(1,0)$ and such that 
$$
\mathfrak a_1(J)_{\mathbb C}=\langle Z_4, \bar Z_4\rangle, \qquad
\big(\tilde\frg_1\big)_{\mathbb C}=\big\langle \tilde Z_k, \tilde{\bar Z}_k \big\rangle_{k=1}^3,
$$
where $\tilde Z_k$ and $\tilde{\bar Z}_k\,(=\overline{\tilde Z}_k)$ are, respectively, the classes of $Z_k$ and $\bar Z_k$ in $\big(\tilde\frg_1\big)_{\mathbb C}$. 
Moreover, $\mathfrak a_1(J)\subseteq\mathcal Z(\frg)$ implies that the  bracket of~$\frg_{\mathbb C}$ can be expressed as 
\begin{equation}\label{corchetes-extension}
\begin{split}
[Z_4,U]_{\frg_{\mathbb C}} &= [\bar Z_4,U]_{\frg_{\mathbb C}}=0, \text{ for any }U\in \frg_{\mathbb C}=\langle Z_k,\bar Z_k\rangle_{k=1}^4,\\
[V,W]_{\frg_{\mathbb C}} &= [\tilde V,\tilde W]_{(\tilde\frg_1)_{\mathbb C}} + \theta(V,W), \text{ for }V, W\in\frh:=\langle Z_k, \bar Z_k\rangle_{k=1}^3,
\end{split}
\end{equation}
for some $\mathbb C$-bilinear map $\theta:\mathfrak h\times\mathfrak h\to\fra_1(J)_{\mathbb C}$.
Note that $\theta$ must satisfy the necessary conditions to define both a Lie bracket (see~\cite[Sections~2 and~3]{SS}) and a well-defined complex structure (see, for instance,~\cite[Theorem 2.8, Chapter IX]{KN}).

By Lemma~\ref{lema1}, we know that $(\tilde\frg_1,\tilde J_1)$ is a $6$-dimensional NLA with 
an SnN complex structure.
Applying~\cite[Proposition~2.3]{UV-SnN} (see also~\cite{COUV}), 
we can choose $\{Z_1, Z_2,  Z_3\}$ so that $\{\tilde Z_1,\tilde Z_2, \tilde Z_3\}$ 
satisfy
\begin{equation}\label{bracket-FamilyIII}
\begin{split}
[\tilde Z_1,\tilde Z_3]_{(\tilde\frg_1)_{\mathbb C}} & =
[\tilde Z_1,\overline{\tilde Z}_3]_{(\tilde\frg_1)_{\mathbb C}} = -\tilde Z_2, \\
[\tilde Z_1,\overline{\tilde Z}_1]_{(\tilde\frg_1)_{\mathbb C}} & = -i\varepsilon (\tilde Z_3+\overline{\tilde Z}_3), \quad
[\tilde Z_1,\overline{\tilde Z}_2]_{(\tilde\frg_1)_{\mathbb C}} = -i\delta (\tilde Z_3 - \overline{\tilde Z}_3),
\end{split}
\end{equation} 
together with the corresponding conjugates, where $\varepsilon\in\{0,1\}$ and $\delta=\pm 1$. Let $\{\tau^k\}_{k=1}^4$ be the basis of $(1,0)$-forms dual to $\{Z^k\}_{k=1}^4$. Then, by~\eqref{corchetes-extension}, \eqref{bracket-FamilyIII} and the formula \eqref{rel-diferencial-corchete}, extended to the complexification, 
we have that the complex structure equations of $(\frg,J)$ in this basis become
$$
d\tau^1=0,\quad
	d\tau^2=\tau^{13}+\tau^{1\bar3},\quad
	d\tau^3=i\,\varepsilon\,\tau^{1\bar1}+i\,\delta\,\tau^{1\bar2} - i\,\delta\,\tau^{2\bar1}, \quad \varepsilon\in\{0,1\},\ \  \delta=\pm 1,
$$
and
\begin{equation*}
	\begin{split}
		d\tau^4 = &\ A_{12}\,\tau^{12}+A_{13}\,\tau^{13}+A_{23}\,\tau^{23}
		+B_{1\bar1}\,\tau^{1\bar 1}+B_{1\bar2}\,\tau^{1\bar 2}+B_{1\bar3}\,\tau^{1\bar 3} \\
		& +B_{2\bar1}\,\tau^{2\bar 1}+B_{2\bar2}\,\tau^{2\bar 2}+B_{2\bar3}\,\tau^{2\bar 3}
		+B_{3\bar1}\,\tau^{3\bar 1}+B_{3\bar2}\,\tau^{3\bar 2}+B_{3\bar3}\,\tau^{3\bar 3},
	\end{split}
\end{equation*}
for some $A_{ij}, B_{r\bar s}\in\mathbb C$, with $1\leq i<j\leq 3$ and $r,s\in\{1,2,3\}$, satisfying $d^2\tau^4=0$. 
Note that $d^2\tau^k=0$ holds for $1\leq k\leq 3$, so the Jacobi identity (which is equivalent to $d^2=0$) is satisfied if and only if $d^2\tau^4=0$. Note also that the above equations agree with \eqref{Nijenhuis-nulo}.

We now compute
\begin{equation*}
\begin{split}
d^2\tau^4 &=
	i\,\big(\delta\,(A_{13}-B_{1\bar3})+\varepsilon\,(A_{23}+B_{2\bar3})\big)\,\tau^{12\bar1}
	+i\,\delta\,(A_{23}-B_{2\bar3})\,\tau^{12\bar2}\\
&+(B_{1\bar2}+B_{2\bar1}+i\,\varepsilon\,B_{3\bar3})\,\tau^{13\bar1}
+(B_{2\bar2}-i\,\delta\,B_{3\bar3})\,\tau^{13\bar2}
-(A_{23}-B_{2\bar3})\,\tau^{13\bar3}\\
&-i\,(\delta\,B_{3\bar1}-\varepsilon\,B_{3\bar2})\,\tau^{1\bar1\bar2}
-(B_{1\bar2}+B_{2\bar1}-i\,\varepsilon\,B_{3\bar3})\,\tau^{1\bar1\bar3}
-(B_{2\bar2}-i\,\delta\,B_{3\bar3})\,\tau^{1\bar2\bar3}\\
&+(B_{2\bar2}+i\,\delta\,B_{3\bar3})\,\tau^{23\bar1}
-i\,\delta\,B_{3\bar2}\,\tau^{2\bar1\bar2}
-(B_{2\bar2}+i\,\delta\,B_{3\bar3})\,\tau^{2\bar1\bar3}
-B_{3\bar2}\,\tau^{3\bar1\bar3}.
\end{split}
\end{equation*}
It is straightforward to check that $d^2\tau^4=0$ is equivalent to 
the following conditions: 
$$B_{2\bar2}=B_{3\bar1}=B_{3\bar2}=B_{3\bar3}=0, \quad B_{2\bar3}=A_{23},$$
and 
$$B_{1\bar3}=A_{13}+2\,\delta\,\varepsilon\,A_{23}, \quad B_{2\bar1}=-B_{1\bar2}.$$
Therefore, the differential of $\tau^4$ becomes
\begin{equation*}
\begin{split}
d\tau^4  = & \ A_{12}\,\tau^{12}+A_{13}\,\tau^{13}+A_{23}\,\tau^{23} \\
&+B_{1\bar1}\,\tau^{1\bar 1}+B_{1\bar2}\,\tau^{1\bar 2}+(A_{13}+2\,\delta\,\varepsilon\,A_{23})\,\tau^{1\bar 3}-B_{1\bar2}\,\tau^{2\bar 1}+A_{23}\,\tau^{2\bar 3}.
\end{split}
\end{equation*}

We next distinguish two cases.
Suppose first that $A_{23}=0$. One can then define a new $(1,0)$-basis for $(\mathfrak g,J)$ as follows:
$$\eta^1=\tau^1,\quad \eta^2=\tau^2,\quad \eta^3=\tau^3,\quad
\eta^4=\tau^4-A_{13}\,\tau^2+i\,\delta\,B_{1\bar2}\,\tau^3.$$
It is easy to check that the complex structure equations in this case are~\eqref{ecus-1} with
$A_\eta=A_{12}$, $B_\eta=B_{1\bar1}-\varepsilon\,\delta\,B_{1\bar2}$ and $\nu=0$.

We now suppose that $A_{23}\neq 0$. The basis of $(1,0)$-forms that gives the structure equations~\eqref{ecus-1}  for $(\mathfrak g,J)$ is
$$\eta^1=\tau^1,\quad \eta^2=\eta^2,\quad \eta^3=\tau^3,\quad
\eta^4=\frac{1}{A_{23}}\, (\tau^4-A_{13}\,\tau^2+i\,\delta\,B_{1\bar2}\,\tau^3),$$
with coefficients
$A_\eta=A_{12}/A_{23}$, $B_\eta=(B_{1\bar1}-\varepsilon\,\delta\,B_{1\bar2})/A_{23}$ and $\nu=1$.
\end{proof}

The previous proposition provides, by construction, the complex structure equations of all the existent $8$-dimensional NLAs with WnN complex structures. However, different choices of the parameters $\varepsilon,\delta,\nu,A_{\eta},B_{\eta}$ might give rise to equivalent complex structures, in the sense of Definition~\ref{def-equivalence}.

\subsection{Reduction of the complex structure equations}\label{subsec:reduction}

The aim here is to reduce the set of possible values that the parameters $A_{\eta},B_{\eta}\in\mathbb C$ in Proposition~\ref{prop-inicial} can take. 
To do so, we separately study the cases $\nu=0$ and $\nu=1$ in order to define an appropriate $(1,0)$-frame for $(\frg,J)$.

\begin{proposition}\label{reduc-nu-0}
Let $(\frg,J)$ be an $8$-dimensional NLA with a WnN complex structure given by~\eqref{ecus-1} with $\nu=0$. Then, there is a $(1,0)$-basis $\{\omega^k\}_{k=1}^4$ such that the structure equations of  
$(\frg,J)$ are~\eqref{structure-eq-WnN} with $\delta=\pm 1$, $\nu=0$ and $(\varepsilon,a,B)$ one of the following:
\begin{itemize}
		\item[i)] $a=B=0$ and $\varepsilon\in\{0,1\}$,
		\item[ii)] $a=0$, $B=1$ and $\varepsilon\in\{0,1\}$,
		\item[iii)] $a=1$, $B\in\{0,1\}$ and $\varepsilon=0$,
		\item[iv)] $a=1$, $B\in[0,\infty)$ and $\varepsilon=1$.
\end{itemize}
\end{proposition}  
\begin{proof}
By hypothesis we have $\nu=0$, so we need to focus on the rest of parameters that appear in the complex structure equations~\eqref{ecus-1}, namely, $\varepsilon\in\{0,1\}$, $\delta=\pm1$, and $A_\eta,B_\eta\in\mathbb C$. In particular, it suffices to distinguish three cases depending on the value of the pair $(A_\eta,B_\eta)$: $A_\eta=B_\eta=0$, $A_\eta=0$ with $B_\eta\neq0$, and $A_\nu\neq0$.

Let us first consider the case $A_\eta=B_\eta=0$. By renaming the basis $\{\eta\}_{k=1}^4$ as $\{\omega\}_{k=1}^4$ one directly gets part i) of the proposition.

Suppose now that the structure equations~\eqref{ecus-1} satisfy $\nu=0$, $A_\eta=0$ and $B_\eta\neq 0$. We define  
$\omega^k=\eta^k$ for $1\leq k\leq 3$ and $\omega^4=\frac{1}{B_\eta}\,\eta^4$. 
A direct computation shows that the differentials of these elements follow~\eqref{structure-eq-WnN} with parameters given by ii).

Finally, let us assume that $\nu=0$ and $A_\eta\neq 0$ in~\eqref{ecus-1}. We write the complex numbers $A_\eta,B_\eta$ in polar form, namely, $A_\eta=a\,e^{i\alpha}$, $B_\eta=b\,e^{i\beta}$, where $a>0$, $b\geq 0$, $\alpha,\,\beta\in[0,2\pi)$. One can then define
$$\omega^1=e^{-\frac{i}{2}\,(\beta-\alpha)}\,\eta^1, \quad
\omega^2=e^{-\frac{i}{2}\,(\beta-\alpha)}\,\eta^2, \quad
\omega^3=\eta^3, \quad
\omega^4=\frac{1}{a}\,e^{-i\beta}\,\eta^4.$$
The structure equations in terms of this basis coincide with~\eqref{structure-eq-WnN}, with coefficients 
$$\nu=0,\quad a=1, \quad B=\frac{b}{a}\geq 0.$$
If $\varepsilon=1$, one directly gets iv). 
If $\varepsilon=0$ and $B_\eta=0$, one obtains part iii) of the statement with $B=0$. If $\varepsilon=0$ and 
$B_\eta\neq 0$, we get iii) with $B=1$ after defining a new $(1,0)$-basis from $\{\omega^k\}_{k=1}^4$ as follows
$$\tilde\omega^1=\omega^1,\quad \tilde\omega^2=\frac{a}{b}\,\omega^2, \quad \tilde\omega^3=\frac{a}{b}\,\omega^3,\quad \tilde\omega^4=\frac{a}{b}\,\omega^4,$$
then omitting the tildes.
\end{proof}

\begin{proposition}\label{reduc-nu-1}
Let $(\frg,J)$ be an $8$-dimensional NLA with a WnN complex structure given by~\eqref{ecus-1} with $\nu=1$. Then, there is a $(1,0)$-basis $\{\omega^k\}_{k=1}^4$ such that the structure equations of  
$(\frg,J)$ are~\eqref{structure-eq-WnN} with $\delta=\pm 1$, $\nu=1$ and $(\varepsilon,a,B)$ one of the following:
	\begin{itemize}
		\item[i)] $a=0$, $B\in\{0,1\}$ and $\varepsilon=0$,
		\item[ii)] $a=0$, $B\in[0,\infty)$ and $\varepsilon=1$,
		\item[iii)] $a=1$, $B\in[0,\infty)$ and $\varepsilon=0$,
		\item[iv)] $a=1$, $B\in\mathbb C$ with $\Imag B> 0$ and $\varepsilon=0$,
		\item[v)] $a\in(0,\infty)$, $B\in\mathbb C$ and $\varepsilon=1$.
	\end{itemize}
\end{proposition}  
\begin{proof}
As $\nu=1$, we will focus on the remaining parameters involved in the structure equations~\eqref{ecus-1}, namely, $\varepsilon\in\{0,1\}$, $\delta=\pm1$, and $A_\eta,B_\eta\in\mathbb C$. 
We distinguish four cases depending on the value of the pair $(A_\eta,\varepsilon)$: $A_\eta=\varepsilon=0$, $A_\eta=0$ with $\varepsilon=1$, $A_\eta\neq0$ with $\varepsilon=0$, and 
$A_\eta\neq0$ with $\varepsilon=1$.

Suppose that $A_\eta=\varepsilon=0$. If one also has $B_\eta=0$ in~\eqref{ecus-1}, part i) of the statement with $B=0$ is directly attained by simply renaming the basis. If $B_\eta\neq 0$, we write
$B_\eta=b\,e^{i\beta}$, with $b>0$ and $\beta\in[0,2\pi)$, and define
$$\omega^1=e^{-i\beta}\,\eta^1,\quad
\omega^2=\frac{1}{\sqrt b}\,e^{-i\beta}\,\eta^2, \quad
\omega^3=\frac{1}{\sqrt b}\,\eta^3,\quad
\omega^4=\frac{1}{b}\,e^{-i\beta}\,\eta^4.$$
This new basis satisfies~\eqref{structure-eq-WnN} with parameters given by i) with $B=1$. 

We now assume $A_\eta=0$ and $\varepsilon=1$ in~\eqref{ecus-1}. Writing $B_\eta=b\,e^{i\beta}$ as before, we can set the new basis of $(1,0)$-forms as
$$\omega^1=e^{-i\,\beta}\,\eta^1,\quad
\omega^2=e^{-i\,\beta}\,\eta^2,\quad
\omega^3=\eta^3,\quad
\omega^4=e^{-i\,\beta}\,\eta^4.$$
The structure equations of $(\frg,J)$ in terms of this basis satisfy ii). 

Let us suppose that the structure equations~\eqref{ecus-1} with $\nu=1$ have $A_\eta\neq 0$ and $\varepsilon=0$. Then, we can write $A_\eta=a\,e^{i\alpha}$ with $a>0$ and $\alpha\in[0,2\pi)$. 
We consider two cases depending on the vanishing of the imaginary part of $B_\eta\,e^{i\alpha}$.  

Firstly, if $\Imag(B_\eta\,e^{i\alpha})= 0$ then
the basis
$$
\omega^1= e^{i\alpha}\,\eta^1, \quad
\omega^2=\frac{e^{i\alpha}}{a}\,\eta^2, \quad
\omega^3=\frac{1}{a}\,\eta^3,\quad
\omega^4=\frac{e^{i\alpha}}{a^2}\,\eta^4,
$$
reduces the structure equations to~\eqref{structure-eq-WnN} with coefficients 
$\nu=a=1$, $B=\frac{1}{a^2}\,B_\eta\,e^{i\alpha}$, and $\varepsilon=0$. In particular, observe that $B\in \mathbb R$. If $B\in[0,\infty)$, one directly gets part iii) of the statement. Otherwise, one can further apply the change of basis $\tilde{\omega}^k= -\omega^k$, for $k=1,3,4$, and $\tilde{\omega}^2= \omega^2$ to get again iii) by simply omitting the tildes.

Secondly, if $\Imag(B_\eta\,e^{i\alpha}) \not= 0$, then
we set $\sigma=1$ when $\Imag(B_\eta\,e^{i\alpha})> 0$, and $\sigma=-1$ otherwise. If we consider the basis
$$
\omega^1=\sigma\,e^{i\alpha}\,\eta^1, \quad
\omega^2=\frac{1}{a}\,e^{i\alpha}\,\eta^2, \quad
\omega^3=\frac{\sigma}{a}\,\eta^3,\quad
\omega^4=\frac{\sigma}{a^2}\,e^{i\alpha}\,\eta^4,
$$
then the corresponding structure equations are~\eqref{structure-eq-WnN} with coefficients 
$$\nu=1,\quad a=1, \quad B=\frac{1}{\sigma\,a^2}\,B_\eta\,e^{i\alpha}, \quad \varepsilon=0.$$
Observe that these parameters satisfy iv). 

Finally, we consider the values $A_\eta\neq 0$ and $\varepsilon=1$ in~\eqref{ecus-1} with $\nu=1$. If we write $A_\eta=a\,e^{i\alpha}$ as before and then set
$$
	\omega^1=e^{i\alpha}\,\eta^1,\quad
	\omega^2=e^{i\alpha}\,\eta^2,\quad
	\omega^3=\eta^3,\quad
	\omega^4=e^{i\alpha}\,\eta^4,
$$
part v)  of the statement is attained. 
\end{proof}

Observe that the combination of Propositions~\ref{reduc-nu-0} and~\ref{reduc-nu-1} provides the first part of Theorem~\ref{main-theorem}. It remains to prove the non-equivalence statement.

\subsection{Classification of complex structures up to equivalence}\label{subsec:classification-complex}

We here complete the proof of Theorem~\ref{main-theorem}, thus showing that the reduction of parameters accomplished in the previous section is optimal (in the sense that it is not possible to find new equivalences among the corresponding complex structures).

Let $(\frg,J)$ and $(\frg',J')$ be two $8$-dimensional NLAs with WnN complex structures. 
Due to the results in Section~\ref{subsec:reduction},
there are $(1,0)$-frames $\{\omega^k\}_{k=1}^4$  for $(\frg,J)$
and $\{\omega'^{\,k}\}_{k=1}^4$ for $(\frg',J')$
in terms of which the complex structure equations of $(\frg,J)$ and $(\frg',J')$ are given by~\eqref{structure-eq-WnN} with parameters $(\varepsilon,\delta,\nu,a,B)$ and $(\varepsilon',\delta',\nu',a',B')$, respectively. Recall that the values of these tuples must follow Proposition~\ref{reduc-nu-0} or Proposition~\ref{reduc-nu-1}.  A direct consequence of Corollary~\ref{equiv-WnN-8D-inicio} is the following one:

\begin{corollary}\label{corolario1}
In the conditions above, if the complex structures are equivalent, then $\varepsilon'=\varepsilon$ and $\delta'=\delta$.
\end{corollary}
\begin{proof}
We note first that from the complex structure equations~\eqref{structure-eq-WnN}, their conjugates, and formula~\eqref{rel-diferencial-corchete},  
one computes
$$\mathcal Z(\mathfrak g)=
	\begin{cases} 
	\langle Z_4+\bar{Z_4},\, i\,(Z_4-\bar Z_4)\rangle, \text{ if }\varepsilon=\nu=1,\\[2pt]
	\langle i\,(Z_3-\bar Z_3),\, Z_4+\bar{Z_4},\, i\,(Z_4-\bar Z_4)\rangle, \text{ if }\varepsilon\nu=0,
	\end{cases}
$$
where $\{Z_k,\bar Z_k\}_{k=1}^4$ is the dual basis of $\{\omega^k,\omega^{\bar k}\}_{k=1}^4$. Therefore, any $8$-dimensional NLA~$\mathfrak g$ with a WnN complex structure $J$ given by~\eqref{structure-eq-WnN} satisfies $\mathfrak a_1(J)=\langle Z_4+\bar{Z_4},\, i\,(Z_4-\bar Z_4) \rangle$. 

By Corollary~\ref{equiv-WnN-8D-inicio}, the induced SnN complex structures $\tilde J_1$ and $\tilde J'_1$ on the respective $6$-dimensional quotient Lie algebras $\tilde\frg_1$ and $\tilde\frg'_1$ are equivalent. Moreover, the complex structure equations of $(\tilde\frg_1,\tilde J_1)$ and $(\tilde\frg'_1,\tilde J'_1)$ are determined by the brackets in \eqref{bracket-FamilyIII} with parameters $(\varepsilon,\delta)$ and $(\varepsilon',\delta')$, respectively. These precisely correspond to the first three equations in~\eqref{structure-eq-WnN}. Now, by~\cite{COUV} (see also~\cite[Proposition~2.3]{UV-SnN}), one has that $\tilde J_1$ and $\tilde J'_1$ are equivalent if and only if $(\varepsilon,\delta)=(\varepsilon',\delta')$.
\end{proof}

In what follows, we will study the equivalence between $J$ and $J'$ from the complexified viewpoint. In other words, 
the isomorphism $f:\frg\to\frg'$ in Definition~\ref{def-equivalence} induces an isomorphism between dual spaces $f^*:\frg'^{\,*}\to\frg^*$ such that $f^*\circ J'=J\circ f^*$, where $J$ and $J'$ still denote the endomorphisms defined on the dual algebras. This in turn gives a $\mathbb C$-linear isomorphism 
\begin{equation}\label{complex-viewpoint}
F:\frg'^{\,1,0}_{J'}\to\frg_J^{1,0} \text{ \ with \ } d\circ F=F\circ d.
\end{equation}  
In fact, $J$ and $J'$ are equivalent if and only if \eqref{complex-viewpoint} is satisfied.
Note that the differential $d$ on the left-hand-side of the equality in~\eqref{complex-viewpoint} corresponds to the Chevalley-Eilenberg differential of~$\frg$, whereas that on the right-hand-side is that of~$\frg'$. They are both denoted by $d$ for the seek of simplicity.

Now, any $\mathbb C$-linear isomorphism $F:\frg'^{\,1,0}_{J'}\to\frg_J^{1,0}$ is determined by a matrix $\Lambda=(\lambda^i_j)_{1\leq i,j\leq 4}\in\textrm{GL}(4,\mathbb C)$ such that
\begin{equation}\label{Fomega}
F\big(\omega'^{\,i}\big)=\sum_{j=1}^4\lambda^i_j\,\omega^j, \text{ for } 1\leq i\leq 4,
\end{equation}
and the condition $d\circ F=F\circ d$ is equivalent to
\begin{equation}\label{cond-Fomega}
\big[d\big(F(\omega'^{\,i})\big)-F(d\omega'^{\,i})\big]_{rs}=0=\big[d\big(F(\omega'^{\,i})\big)-F(d\omega'^{\,i})\big]_{r\bar{s}}\, , 
\end{equation}
for any $1\leq i\leq 4$ and $1\leq r,s\leq 4$, where $\big[d\big(F(\omega'^{\,i})\big)-F(d\omega'^{\,i})\big]_{rs}$ denotes the coefficient of $\omega^{rs}$ in the expression $d\big(F(\omega'^{\,i})\big)-F(d\omega'^{\,i})$ (similarly for the coefficient of $\omega^{r \bar s}$).

The purpose of the following lemmas is to check that the existence of any $F$ defined as above yields $(\varepsilon,\delta,\nu,a,B)=(\varepsilon',\delta',\nu',a',B')$, thus showing the non-equivalence of complex structures determined by different sets of parameters. 
The next result allows to simplify the expression~\eqref{Fomega}. 

\begin{lemma}\label{lemma1-Fomega}
The forms $F(\omega'^{\,i})\in\frg_J^{1,0}$ satisfy the following conditions:
$$F(\omega'^{\,1})\wedge\omega^1=0, \quad 
F(\omega'^{\,2})\wedge\omega^{12}=0,\quad
F(\omega'^{\,3})\wedge\omega^{3}=0,\quad
F(\omega'^{\,4})\wedge\omega^{124}=0.$$
In particular, the matrix $\Lambda=(\lambda^i_j)_{1\leq i,j\leq 4}$ that determines $F$ is triangular, so
$$\Pi_{k=1}^4\lambda^k_k=\det{\Lambda}\neq 0.$$
\end{lemma}

\begin{proof}
Due to Corollary~\ref{corolario1}, the tuples $(\varepsilon,\delta,\nu,a,B)$ and $(\varepsilon',\delta',\nu',a',B')$ that determine $\frg_J^{1,0}$ and $\frg'^{\,1,0}_{J'}$ satisfy $\varepsilon'=\varepsilon$ and $\delta'=\delta$.
A direct computation of the condition \eqref{cond-Fomega} for $i=1$ gives $\lambda^1_2=\lambda^1_3=0$.  
Moreover, the condition \eqref{cond-Fomega} for $i=2$ implies that
$$0=\big[d\big(F(\omega'^{\,2})\big)-F(d\omega'^{\,2})\big]_{4\bar{k}}=\lambda^1_4\,\bar{\lambda}^3_k, \text{ for }k=1,2,3,4.$$
As the matrix $\Lambda=(\lambda^i_j)_{1\leq i,j\leq 4}$ is non-singular, one necessarily has
$\lambda^1_4=0$. This gives $F(\omega'^{\,1})\wedge\omega^1=0$. 
Note that $\lambda^1_1\neq 0$ to ensure $\Lambda\in\textrm{GL}(4,\mathbb C)$.

Now, we analyze the condition \eqref{cond-Fomega} for $i=2,3$ considering that $\lambda^1_k=0$ for $k=2,3,4$. We first observe that
\begin{equation*}
\begin{split}
0=\big[d\big(F(\omega'^{\,2})\big)-F(d\omega'^{\,2})\big]_{2\bar1}&=-i\,\delta\lambda^2_3,\\
0=\big[d\big(F(\omega'^{\,3})\big)-F(d\omega'^{\,3})\big]_{4\bar1}&=i\,\delta\lambda^2_4\bar{\lambda}^1_1.\\
\end{split}
\end{equation*}
As $\delta\in\{-1,1\}$ and $\lambda^1_1\neq 0$, one immediately has  $F(\omega'^{\,2})\wedge\omega^{12}=0$. Moreover, one can check that
\begin{equation*}
\begin{split}
0=\big[d\big(F(\omega'^{\,2})\big)-F(d\omega'^{\,2})\big]_{14}&=-\lambda^1_1\lambda^3_4,\\
0=\big[d\big(F(\omega'^{\,2})\big)-F(d\omega'^{\,2})\big]_{1\bar1}&=
	-\lambda^1_1\bar{\lambda}^3_1+i\,\varepsilon\,\lambda^2_3+B\,\lambda^2_4=
	-\lambda^1_1\bar{\lambda}^3_1,\\
0=\big[d\big(F(\omega'^{\,2})\big)-F(d\omega'^{\,2})\big]_{1\bar2}&=
	-\lambda^1_1\bar{\lambda}^3_2+i\,\delta\,\lambda^2_3=
	-\lambda^1_1\bar{\lambda}^3_2,
\end{split}
\end{equation*}
from where we conclude that $F(\omega'^{\,3})\wedge\omega^{3}=0$. The last condition in the statement comes directly when we annihilate $\big[d\big(F(\omega'^{\,4})\big)-F(d\omega'^{\,4})\big]_{1\bar2}$.
\end{proof}

We are now able to describe any isomorphism between $(\frg,J)$ and $(\frg',J')$.

\begin{lemma}\label{lemma2-Fomega}
The complex structures $J$ and $J'$ are equivalent if and only if the isomorphism $F:\frg'^{\,1,0}_{J'}\to\frg^{1,0}_{J}$ given by~\eqref{Fomega} satisfies 
the conditions in Lemma~\ref{lemma1-Fomega} together with
$$\lambda^1_1=e^{i\theta}, \quad 
\lambda^2_2=\lambda\,e^{i\theta}, \quad 
\lambda^3_3=\lambda,$$
$$\varepsilon\,(1-\lambda)=0, \quad
\Imag(e^{i\theta}\,\bar\lambda^2_1)=0, \quad
\lambda\,e^{i\,\theta}\,\lambda^4_2-\nu\,\lambda^2_1\lambda^4_4=0,$$
where $\theta\in[0,2\pi)$ and $\lambda\in\mathbb R^*$.
Moreover, the parameters that determine $J$ and $J'$ are related as follows:
$$\varepsilon'=\varepsilon,\quad \delta'=\delta, \quad 
a'=\frac{a\,\lambda^4_4}{(\lambda^1_1)^2\,\lambda^3_3}, \quad  
B'=B\,\lambda^4_4, \quad 
\nu'=\frac{\nu\,\lambda^4_4}{\lambda^1_1\,(\lambda^3_3)^2}.$$
\end{lemma}

\begin{proof}
We first observe that the equalities $\varepsilon'=\varepsilon$ and $\delta'=\delta$ come directly from Corollary~\ref{corolario1}. We also recall that $\varepsilon\in\{0,1\}$ and $\delta\in\{-1,1\}$.

One can check that the conditions for $F$ given in Lemma~\ref{lemma1-Fomega} imply that~\eqref{cond-Fomega} holds for $i=1$. We now compute
\begin{equation*}
\begin{split}
\big[d\big(F(\omega'^{\,2})\big)-F(d\omega'^{\,2})\big]_{13}&
	= \lambda^2_2-\lambda^1_1\,\lambda^3_3,\\
\big[d\big(F(\omega'^{\,2})\big)-F(d\omega'^{\,2})\big]_{1\bar3}&
	=-\lambda^1_1\,\bar\lambda^3_3+\lambda^2_2,\\
\big[d\big(F(\omega'^{\,3})\big)-F(d\omega'^{\,3})\big]_{2\bar1}&
	=i\,\delta\,(\lambda^2_2\,\bar\lambda^1_1-\lambda^3_3).
\end{split}
\end{equation*}
All these expressions must be equal to zero, so from the first one we get $\lambda^2_2=\lambda^1_1\lambda^3_3$, and replacing this value in the other two we have
$$\lambda^1_1\,(\lambda^3_3-\bar\lambda^3_3)=0, \qquad
|\lambda^1_1|^2-1=0.$$
Since $\lambda^k_k\neq 0$ for every $1\leq k\leq 4$ due to Lemma~\ref{lemma1-Fomega}, one necessarily has $\lambda^3_3=\lambda\in\mathbb R^*$. Moreover, $|\lambda^1_1|=1$. One can check that these equalities ensure the condition~\eqref{cond-Fomega} for $i=2$, whereas for $i=3$ it is equivalent to
\begin{equation*}
\begin{split}
0=\big[d\big(F(\omega'^{\,3})\big)-F(d\omega'^{\,3})\big]_{1\bar1}
	&=
	i\,\big(\varepsilon\,\lambda-\varepsilon-\delta(\lambda^1_1\bar\lambda^2_1-\bar\lambda^1_1\lambda^2_1)\big)\\
	&=2\,\delta\,\Imag(\lambda^1_1\bar\lambda^2_1)+i\,\varepsilon\,(\lambda-1).
\end{split}
\end{equation*}
The imaginary part of this expression gives the condition
$\varepsilon\,(1-\lambda)=0$
in the statement of this lemma. 
Moreover, since $|\lambda^1_1|=1$ we can write $\lambda^1_1=e^{i\,\theta}$ for $\theta\in[0,2\pi)$, and then the real part of the expression above provides the equation $\Imag(e^{i\theta}\,\bar\lambda^2_1)=0$. The remaining equalities of the statement come from the condition~\eqref{cond-Fomega} applied to $i=4$. More precisely, $\nu', a'$ and $B'$ can be solved from
\begin{equation*}
\begin{split}
0=\big[d\big(F(\omega'^{\,4})\big)-F(d\omega'^{\,4})\big]_{12}
	&=-a'\,(\lambda^1_1)^2\lambda^3_3+a\,\lambda^4_4,\\
0=\big[d\big(F(\omega'^{\,4})\big)-F(d\omega'^{\,4})\big]_{23}
	&=-\nu'\,\lambda^1_1\,(\lambda^3_3)^2+\nu\,\lambda^4_4,\\
0=\big[d\big(F(\omega'^{\,4})\big)-F(d\omega'^{\,4})\big]_{1\bar1}
	&=-B'+B\,\lambda^4_4.\\
\end{split}
\end{equation*}
Then, \eqref{cond-Fomega} for $i=4$ is equivalent to
\begin{equation*}
\begin{split}
0=\big[d\big(F(\omega'^{\,4})\big)-F(d\omega'^{\,4})\big]_{13}
	&=\frac{1}{\lambda\,\lambda^1_1}\,(\lambda\,\lambda^1_1\,\lambda^4_2
			-\nu\,\lambda^2_1\,\lambda^4_4),\\
0=\big[d\big(F(\omega'^{\,4})\big)-F(d\omega'^{\,4})\big]_{1\bar3}
	&=\frac{1}{\lambda\,\lambda^1_1}\,
	\big(\lambda\,\lambda^1_1\,\lambda^4_2-\nu\,\lambda^2_1\,\lambda^4_4
			-2\,\delta\,\varepsilon\,\nu\,(1-\lambda)\,\lambda^1_1\,\lambda^4_4\big).
\end{split}
\end{equation*}
Observe that the first equation gives the remaining condition in the statement of the lemma. The second one is trivially satisfied due to the aforementioned first equation together with the condition $\varepsilon(1-\lambda)=0$.
\end{proof}

We now make use of this result to classify, up to equivalence, weakly non-nilpotent complex structures in dimension $8$.

\begin{proposition}
Distinct choices of parameters $(\varepsilon,\delta,\nu,a,B)$ in Propositions~\ref{reduc-nu-0} and~\ref{reduc-nu-1} provide non-equivalent complex structures.
\end{proposition}

\begin{proof}
Let $(\frg,J)$ and $(\frg',J')$ be nilpotent Lie algebras with WnN complex structures given by~\eqref{structure-eq-WnN} with parameters $(\varepsilon,\delta,\nu,a,B)$ and $(\varepsilon',\delta',\nu',a',B')$, respectively.
Note that any $F$ providing an equivalence between $J$ and $J'$ satisfies Lemma~\ref{lemma2-Fomega}. Since $\nu,\nu'\in\{0,1\}$, we first observe that the relation $\nu'=\nu\,\lambda^4_4/(\lambda^1_1\,(\lambda^3_3)^2)$
naturally gives $\nu'=\nu$. 
As a consequence, we can separately study the cases $\nu=\nu'=0$ and $\nu=\nu'=1$. In other words, the complex structures in Proposition~\ref{reduc-nu-0} are not equivalent to those in Proposition~\ref{reduc-nu-1}.

Let $\nu=\nu'=0$. Since $\delta'=\delta$ and $\varepsilon'=\varepsilon$ by Lemma~\ref{lemma2-Fomega}, we essentially need to focus on $(a,B)$ and $(a',B')$ and check that $(a',B')=(a,B)$ if the complex structures are equivalent. We will make use of the relations between these parameters provided in Lemma~\ref{lemma2-Fomega}, namely,
\begin{equation}\label{rel-a-B}
a'=\frac{a\,\lambda^4_4}{(\lambda^1_1)^2\,\lambda^3_3}, \quad  
B'=B\,\lambda^4_4.
\end{equation}
Recall that the possible values of $(\varepsilon,\delta,\nu=0,a,B),\, (\varepsilon'=\varepsilon,\delta'=\delta,\nu'=0,a',B')$ are listed in Proposition~\ref{reduc-nu-0}. In particular, one has $a,a'\in\{0,1\}$, so two cases arise: if $a=0$, then~\eqref{rel-a-B} immediately gives $a'=0$; if $a=1$ then $a'\neq 0$, as by Lemma~\ref{lemma1-Fomega} one has $\Pi_{k=1}^4\lambda^k_k=\det{\Lambda}\neq 0$, so the only possibility is having $a'=1$. Consequently, we can conclude that $a=a'$. Similarly, one can prove that $B=0$ always implies $B'=0$, and $B=1$ implies $B'=1$ when either $a=a'=0$ or $a=a'=1$, $\varepsilon=\varepsilon'=0$. Therefore, we just need to focus on the case $a=1$, $B\in(0,\infty)$ and $\varepsilon=1$. Since $a'=a=1$, we see from~\eqref{rel-a-B} that $\lambda^4_4=(\lambda^1_1)^2\,\lambda^3_3$, where
$$\lambda^1_1=e^{i\theta}, \text{ with }\theta\in[0,2\pi), \quad 
\lambda^3_3=1,$$
by Lemma~\ref{lemma2-Fomega}. Therefore, $B'=B\,e^{2i\theta}=B\big(\cos(2\theta)+i\sin(2\theta)\big)$. As $B'$ is a real number, one has that $\theta\in\{0,\pi/2,\pi,3\pi/2\}$. Hence, $B'=\pm B$. By Proposition~\ref{reduc-nu-0} we know that $B'$ has to be a non-negative number, so $B'=B$. This concludes the proof for the case $\nu=\nu'=0$.

Suppose now that $\nu=\nu'=1$. By Lemma~\ref{lemma2-Fomega}, we have $\delta'=\delta$ and $\varepsilon'=\varepsilon$, but also $\lambda^4_4=\lambda^1_1(\lambda^3_3)^2$ to ensure the condition $\nu'=\nu=1$. As a consequence, the relations between $(a',B')$ and $(a,B)$ become
\begin{equation}\label{rel-a-B-2}
a'= \frac{a\,\lambda^3_3}{\lambda^1_1},
\quad  
B'=B\,\lambda^1_1(\lambda^3_3)^2,
\end{equation}
where $\lambda^1_1,\lambda^3_3$ satisfy the conditions in Lemma~\ref{lemma2-Fomega}, in particular 
$$\lambda^1_1=e^{i\theta}, \text{ with }\theta\in[0,2\pi), \quad 
\lambda^3_3=\lambda\in\mathbb R^*, \quad \varepsilon\,(1-\lambda)=0.$$
Moreover, recall that the possible values of $(\varepsilon,\delta,\nu=1,a,B),\, (\varepsilon'=\varepsilon,\delta'=\delta,\nu'=1,a',B')$ are listed in Proposition~\ref{reduc-nu-1}. In particular, let us remark that $a,a'$ are non-negative real numbers, so we need $\theta=0$ or $\theta=\pi$, namely, $\lambda^1_1=\pm 1$. Hence,~\eqref{rel-a-B-2} becomes
\begin{equation}\label{rel-a-B-2-2}
a'=\pm\,a\,\lambda, \quad  
B'=\pm\,B\,\lambda^2, \text{ \ where }\varepsilon\,(1-\lambda)=0.
\end{equation}

We note that if $\varepsilon=\varepsilon'=1$, then $\lambda=1$ and thus 
$a'=\pm\,a$, $B'=\pm\,B$, where $a,a'\in[0,\infty)$ due to Proposition~\ref{reduc-nu-1}. We now observe the following: if $a=0$, then $a'=0$ and by Proposition~\ref{reduc-nu-1} one has $B,B'\in[0,\infty)$, so $B'=B$; if $a\neq 0$, then also $a'\neq 0$ and $a,a'\in (0,\infty)$, which requires $a'=a$ and $B'=B$. 

Let us now suppose that $\varepsilon=\varepsilon'=0$. We first notice that, as a consequence of Proposition~\ref{reduc-nu-1}, we can assume $a,a'\in\{0,1\}$. It is then clear from~\eqref{rel-a-B-2-2} that $a'=a$. In particular, if $a=0$ then $a'=0$ and, as $B,B'\in\{0,1\}$, one can conclude that either $B=B'=0$ or $B=1$, $B'=\pm\lambda^2\neq 0$, which implies $B'=B=1$. If $a=1$ then $a'=\pm\lambda\neq 0$, so $a'=1$ and $\lambda=\pm 1$. 
As a consequence, we get $B'=\pm B$. 
Note that $B'$ is a real number if and only if $B$ is, so the complex structures given by the case~$\textrm{iii)}$ of Proposition~\ref{reduc-nu-1} are not equivalent to those given by the case~$\textrm{iv)}$. 
Moreover, case $\textrm{iii)}$ requires that $B,B' \in[0,\infty)$ so the equality $B'=\pm B$  holds if and only if $B'=B$. 
Finally, if the two complex structures are given by the case $\textrm{iv)}$ of Proposition~\ref{reduc-nu-1}, then $\mathfrak{Im}B, \mathfrak{Im}B' > 0$ and the condition $B'=\pm B$ can only be satisfied when the plus sign is taken in the expression of $B'$, namely, $B'=B$.
\end{proof}

\section{Classification of 8-dimensional nilpotent Lie algebras\\ 
with WnN complex structures}\label{clasi-real}

\noindent The aim of this section is to classify the $8$-dimensional NLAs that admit WnN complex structures. We first identify the real NLAs associated to the complex structure equations given in Theorem~\ref{main-theorem}. Then, we show that these Lie algebras are not isomorphic. The final result is the following one: 

\begin{theorem}\label{main-theorem-2}
Let $\mathfrak g$ be an $8$-dimensional NLA. There is a WnN complex structure on $\mathfrak g$ if and only if $\mathfrak g$ is isomorphic to one (and only one) of the following Lie algebras:
\begin{equation*}
	\begin{split}
	& \mathfrak f_1= (0,\,0,\,0,\,12,\,23,\,14-35,\,0,\,0), \\
	& \mathfrak f_2=(0,\,0,\,12,\,13,\,23,\,14+25,\,0,\,0), \\
	& \mathfrak f_3 = (0,\,0,\,0,\,12,\,13,\,23,\,15+26,\,0),\\
	& \mathfrak f_4^\gamma = (0,\,0,\,\gamma\cdot 12,\,13,\,14,\,23,\,26,\,16+24),\\
	& \mathfrak f_5^\gamma = (0,\,0,\,0,\,13,\,23,\,34,\,\frac{\gamma}{2}\cdot 12+35,\,14+25),\\
	& \mathfrak f_6 = (0,\,0,\,12,\,13,\,23,\,14+25,\,16+35,\,26-34),\\
	& \mathfrak f_7 = (0,\,0,\,0,\,13,\,23,\,14+25,\, 2\cdot 14+34,\, 15+24+35),\\
	& \mathfrak f_8=(0,\,0,\,12,\,13,\,23,\,14+25,\,2\cdot 14-26+34,\,15+16+24+35),
	\end{split}
\end{equation*}
where $\gamma\in\{0,1\}$.
\end{theorem}

Let us recall that every $8$-dimensional NLA $\mathfrak g$ with a WnN complex structure is determined by the complex structure equations~\eqref{structure-eq-WnN}, where the set of parameters $(\varepsilon,\delta,\nu,a,B)$ satisfies Theorem~\ref{main-theorem} (see also Propositions~\ref{reduc-nu-0} and~\ref{reduc-nu-1}). 
In Table~\ref{tab:cambios-base-real} we define explicit real bases $\{e^k\}_{k=1}^8$ that allow to arrive directly from the complex structure equations to the real Lie algebras listed in Theorem~\ref{main-theorem-2}, depending on the value of the tuple $(\varepsilon,\delta,\nu,a,B)$. In this table, we simply write $(\varepsilon,\nu,a,B)$ instead of $(\varepsilon,\delta,\nu,a,B)$, because one always has $\delta\in\{-1,1\}$. We observe that the value $\nu=0$ gives rise to the Lie algebras $\mathfrak f_1$, $\mathfrak f_2$, $\mathfrak f_3$, and $\mathfrak f_4^{\,\gamma}$ in Theorem~\ref{main-theorem-2}, whereas $\nu=1$ provides the Lie algebras $\mathfrak f_5^{\gamma}$, $\mathfrak f_6$, $\mathfrak f_7$, and $\mathfrak f_8$, where $\gamma\in\{0,1\}$. 
Moreover, one can go the other way round, in the sense that Table~\ref{tab:cambios-base-real} allows to construct the whole space of WnN complex
structures on each Lie algebra up to equivalence.

\begin{table}[h]
\centering
\renewcommand{\arraystretch}{1.2}
\renewcommand{\tabcolsep}{2pt}
\begin{tabular}{|c|l|c|}
\hline
$(\varepsilon,\nu,a,B)$ & \ Real basis $\{e^k\}_{k=1}^8$ & NLA \\
\hline\hline
$(0, 0, 0, 0)$
& $\begin{array}{ll}
	\omega^1 = e^1-i\,e^3, & \quad\omega^3 = \frac12\,e^2+2\,i\,\delta\,e^6,\\
	\omega^2 =e^4+i\,e^5, & \quad\omega^4 =e^7-i\,e^8.\\
	\end{array}$
& $\mathfrak f_1$ \\
\hline\hline
$(1, 0, 0, 0)$
& $\begin{array}{ll}
	\omega^1 = \frac{\sqrt 2}{2}\,(e^1+i\,e^2), & \quad\omega^3 =e^3+2\,i\,\delta\,e^6,\\
	\omega^2 =\sqrt 2\,(e^4+i\,e^5), & \quad\omega^4 =e^7-i\,e^8.\\
	\end{array}$
& $\mathfrak f_2$ \\
\hline\hline
$(0, 0, 0, 1)$
& $\begin{array}{ll}
	\omega^1 = e^1 + i\,e^2, & \quad\omega^3 = 2\,(e^3+4\,i\,\delta\,e^7),\\
	\omega^2 = 4\,(e^5+i\,e^6), & \quad\omega^4 = 2\,(e^8-i\,e^4).\\
	\end{array}$
& $\mathfrak f_3$ \\
\hline\hline
$(1, 0, 0, 1)$
& $\begin{array}{ll}
	\omega^1 = e^1 + i\,e^2, &  \quad\omega^3 = 2\,(e^3+4\,i\,\delta\,e^6),\\
	\omega^2 = 4\,(e^4+i\,e^5), & \quad\omega^4 = 2\,(e^7-i\,e^3+i\,e^8).\\
	\end{array}$
& $\mathfrak f_2$ \\
\hline\hline
$\begin{array}{c}
(0, 0, 1, B\in\{0,1\})\\
(1, 0, 1, B\in\mathbb R^{\geq 0})
\end{array}$
& $\begin{array}{l}
	\omega^1 = e^1 + i\,e^2,\\ 
	\omega^2 = 2\,\big(2\,e^4+2\,i\,e^6+i\,B\,(1-\varepsilon)\,e^2\big), \\
	\omega^3 = 2\,\big(e^3+4\,i\,\delta\,(e^5+e^7)\big), \\ 
	\omega^4 = 4\,\big(e^5-e^7+i\,(e^8-\varepsilon\,\frac{B}{2}\,e^3)\big).
	\end{array}$
& $\mathfrak f_4^{\,\varepsilon}$ \\
\hline\hline
$(0, 1, 0, B\in\{0,1\})$
& $\begin{array}{ll}
	\omega^1 = e^1 + i\,e^2, &  \quad\omega^3 = e^3+4\,i\,\delta\,e^8,\\
	\omega^2 = 2\,(e^4+i\,e^5), & \quad\omega^4 = -4\,(e^6+i\,e^7).\\
	\end{array}$
& $\mathfrak f_5^{\,B}$ \\
\hline\hline
$(1, 1, 0, B\in\mathbb R^{\geq 0})$
& $\begin{array}{l}
	\omega^1 = e^1 + i\,e^2,\\ 
	\omega^2 = 4\,(e^4+i\,e^5), \\
	\omega^3 = 2\,(e^3+4\,i\,\delta\,e^6), \\ 
	\omega^4 = 16\,(e^8-i\,e^7)+4\,\delta\,(e^4+i\,e^5)-2\,i\,B\,e^3.
	\end{array}$
& $\mathfrak f_6$ \\
\hline\hline
$(0, 1, 1, B\in\mathbb R^{\geq 0} )$
& $\begin{array}{l}
	\omega^1 = -(e^1 + i\,e^2),\\ 
	\omega^2 = -(e^4-2\,B\,e^1 +i\,e^5), \\
	\omega^3 = \,(\frac12\,e^3+2\,i\,\delta\,e^6), \\ 
	\omega^4 = (e^7-e^6)+2\,B (e^4-2\,B\,e^1)+i\, e^8.
	\end{array}$
& $\mathfrak f_7$ \\
\hline\hline
$\begin{array}{l}
	(0, 1, 1, \Imag B>0)\\ 
	B=b_1+i\, b_2\\ 
	b_2>0	
	\end{array}$
& $\begin{array}{l}
	\omega^1 = -2\,b_2\,(e^1 + i\,e^2),\\ 
	\omega^2 = 4\,b_2 \big(b_1\,e^1 +\frac{b_2}{2}\,e^2 -b_2\,e^4 +i\,b_2 (\frac12\,e^1 - e^5) \big), \\
	\omega^3 = b_2\,(e^3+16\,i\,\delta\,b_2^2\,e^6), \\ 
	\omega^4 = -8\,b_2\big( b_1(b_1\,e^1 +\frac{b_2}{2}\,e^2 - b_2\,e^4) 
		-b_2^2( \frac12 e^5 \!-\! e^6 + e^7)  \\ 
	\quad\qquad\qquad\   -i\,b_2^2\, ( \frac12 e^4+e^8) \big). 
	\end{array}$
&  $\mathfrak f_7$ \\
\hline\hline
$(1, 1, a>0, B\in\mathbb C)$
& $\begin{array}{l}
	\omega^1 = -\frac{a}{4}\,(e^1+i\,e^2),\\ 
	\omega^2 = -\frac{a^3}{16}\,(e^4+i\,e^5), \\
	\omega^3 = \frac{a^2}{8}\,\big(e^3+i\,\delta\,\frac{a^2}{4}\,e^6\big), \\ 
	\omega^4 = \frac{a^5}{64}\,(e^7-e^6+i\,e^8)-\frac{i\,a^2\,B}{8}\,e^3-\delta\,\frac{a^3}{16}\,(e^4+i\,e^5).
	\end{array}$
& $\mathfrak f_8$ \\
\hline
\end{tabular}
\medskip
\caption{Real nilpotent Lie algebras and complex structures}
\label{tab:cambios-base-real}
\end{table}

\begin{remark}
When $\nu=a=B=0$ in~\eqref{structure-eq-WnN}, there is a product of a $6$-dimensional NLA with an SnN complex structure and the $2$-dimensional Abelian Lie algebra. Hence, we follow the proof of~\cite[Proposition 2.4]{UV-SnN} to define the corresponding real basis $\{e^k\}_{k=1}^8$ in Table~\ref{tab:cambios-base-real}. 
In particular, note that
$$\mathfrak f_1= \frh^-_{19}\times\mathbb R^2, \qquad \mathfrak f_2= \frh^+_{26}\times\mathbb R^2,$$ 
where $\mathfrak h^-_{19}= (0,\,0,\,0,\,12,\,23,\,14-35)$ and $\frh^+_{26}=(0,\,0,\,12,\,13,\,23,\,14+25)$. We also notice that $\mathfrak f_1$ and $\mathfrak f_2$ are, together with $\mathfrak f_3$, the only decomposable Lie algebras in Theorem~\ref{main-theorem-2}. 
\end{remark}

\begin{remark}
There are two sets of complex parameters 
$(\varepsilon,\delta,\nu,a,B)$ in Table~\ref{tab:cambios-base-real} whose underlying Lie algebra is $\mathfrak f_2$, namely, $(1,\pm1, 0, 0, 0)$ 
and $(1,\pm 1, 0, 0, 1)$. The former corresponds to the WnN structures on $\mathfrak f_2=\mathfrak h_{26}^+\times\mathbb R^2$ that are products of complex structures
on $\mathfrak h_{26}^+$ and 
$\mathbb R^2$. The latter provides totally new complex structures on the direct product~$\mathfrak h_{26}^+\times\mathbb R^2=\mathfrak f_2$.
\end{remark}

\begin{remark}
Recall that the complex parameter $\delta$ provides two non-equivalent complex structures for 
$\delta=-1$ and $\delta=1$. Consequently, each of the Lie algebras $\mathfrak f_1$, $\mathfrak f_3$, 
$\mathfrak f_5^0$, and~$\mathfrak f_5^1$ admits exactly two non-equivalent WnN complex structures, and the Lie algebras $\mathfrak f_2$ and $\mathfrak f_4^0$ have four. In contrast, there is an infinite number of non-equivalent WnN complex structures on $\mathfrak f_4^1$, $\mathfrak f_6$, $\mathfrak f_7$ and $\mathfrak f_8$.
\end{remark}

To finish the proof of Theorem~\ref{main-theorem-2} one has to check that the Lie algebras $\mathfrak f_1, \,\mathfrak f_2,\,\ldots,\mathfrak f_8$ are pairwise non-isomorphic. The following invariants will be specially useful for this purpose:
\begin{itemize}
\item Ascending type: it is the $s$-tuple $(m_1,\ldots, m_s)$ associated to the ascending central series~\eqref{ascending-central-series} of an $s$-step NLA $\mathfrak g$, defined by $m_k=\text{dim}\,\frg_k$, for $k=1,\ldots,s$. 
\item Descending type: similar to the above, but making use of the descending central series $\{\mathfrak g^k\}_{k\geq 1}$ of $\mathfrak g$ instead, which is given by $\mathfrak g^0=\mathfrak g$ and $\mathfrak g^k=[\mathfrak g^{k-1},\mathfrak g]$, for $k\geq 1$.
\item Betti numbers $b_k$: they correspond to the dimensions of the Chevalley-Eilenberg cohomology groups of the Lie algebra $\mathfrak g$, namely,
$$H^k(\mathfrak g)=
	\frac{\text{Ker}\big\{ d:\bigwedge^k\mathfrak g^*\to\bigwedge^{k+1}\mathfrak g^* \big\}}
	{\text{Im}\big\{ d:\bigwedge^{k-1}\frg^*\to\bigwedge^k\frg^* \big\}},
	\quad 0\leq k\leq\text{dim}\,\frg.$$
\item Number $n_d$: maximum number of linearly independent  decomposable $d$-exact $2$-forms $\alpha$, i.e. $\alpha=d\gamma=\beta_1\wedge\beta_2$ for some $\beta_1,\beta_2,\gamma\in\frg^*$.
\end{itemize}

Since the Lie algebras in Theorem~\ref{main-theorem-2} are discrete, the invariants above can be calculated using an appropriate mathematical software.
We provide their values in Table~\ref{tab:real-invariants}.

\begin{table}[h]
\centering
\renewcommand{\arraystretch}{1.4}
\renewcommand{\tabcolsep}{8pt}
\begin{tabular}{|c|c|c|cccc|c|}
\hline
NLA & Ascending type & Descending type & $b_1$ & $b_2$ & $b_3$ & $b_4$ & $n_d$ \\
\hline\hline
$\mathfrak f_1$ & $(3,5,8)$ & $(8,3,1)$ & 5 & 12 & 19 & 22 & 2 \\
\hline
$\mathfrak f_2$ & $(3,5,6,8)$ & $(8,4,3,1)$ & 4 & 9 & 16 & 20 & 3 \\
\hline
$\mathfrak f_3$ & $(3,5,8)$ & $(8,4,1)$ & 4 & 10 & 18 & 22 & 3 \\
\hline
$\mathfrak f_4^0$ & $(3,5,8)$ & $(8,5,3)$ & 3 & 7 & 13 & 16 & 4 \\
\hline
$\mathfrak f_4^1$ & $(3,5,6,8)$ & $(8,6,5,3)$ & 2 & 6 & 13 & 16 & 5 \\
\hline
$\mathfrak f_5^0$ & $(3,5,8)$ & $(8,5,3)$ & 3 & 7 & 14 & 18 & 4 \\
\hline
$\mathfrak f_5^1$ & $(3,5,8)$ & $(8,5,3)$ & 3 & 7 & 14 & 18 & 3 \\
\hline
$\mathfrak f_6$ & $(2,3,5,6,8)$ & $(8,6,5,3,2)$ & 2 & 3 & 6 & 8 & 3 \\
\hline
$\mathfrak f_7$  & $(3,5,8)$ & $(8,5,3)$ & 3 & 7 & 13 & 16 & 3   \\
\hline
$\mathfrak f_8$ & $(2,3,5,6,8)$ & $(8,6,5,3,2)$ & 2 & 3 & 6 & 8 & 3 \\
\hline
\end{tabular}
\medskip
\caption{Some invariants of the NLAs listed in Theorem~\ref{main-theorem-2} 
}
\label{tab:real-invariants}
\end{table}

A direct consequence of Table~\ref{tab:real-invariants} is that most of the NLAs listed in Theorem~\ref{main-theorem-2} are not isomorphic. 
Indeed, the only exceptions are the NLAs $\mathfrak f_6$ and $\mathfrak f_8$, for which all the previous invariants coincide.
Therefore, to complete the proof of Theorem~\ref{main-theorem-2} we need to show the following:

\begin{lemma}
	The Lie algebras $\mathfrak f_6$ and $\mathfrak f_8$ are not isomorphic.
\end{lemma}
\begin{proof}
Let us show that the number of functionally independent generalized Casimir operators is different for the NLAs $\mathfrak f_6$ and $\mathfrak f_8$.

Let $\mathfrak g$ be a Lie algebra of dimension $m$ defined, in terms of some basis $\{x_k\}_{k=1}^m$, by the Lie brackets $[x_i,x_j]=\sum_{k=1}^mc^k_{ij} x_k$, for $1\leq i<j\leq m$. The vectors
$$\hat{X}_k=\sum_{i,j=1}^mc^j_{ki} x_j\frac{\partial}{\partial x_i}, \quad 1\leq k\leq m,$$
generate a basis of the coadjoint representation of $\mathfrak g$. If we consider the matrix $C_{\mathfrak g}$ constructed by rows from the coefficients of these vectors, the number of functionally independent generalized Casimir operators, denoted by $n_I(\mathfrak g)$, can be computed as $n_I(\mathfrak g)=m-\text{rank}\,C_{\mathfrak g}$ (see \cite{SW} for further details).

We first use~\eqref{rel-diferencial-corchete} to find the Lie brackets of $\mathfrak f_6$ from the structure equations of this algebra given in Theorem~\ref{main-theorem-2}. Then, from the expressions of the vectors 
$\hat{X}_1,\ldots, \hat{X}_8$, one finds
$$C_{\mathfrak f_6}=\begin{pmatrix}
0 & -x_3 & -x_4 & -x_6 & 0 & -x_7 & 0 & 0 \\
x_3 & 0 & -x_5 & 0 & -x_6 & -x_8 & 0 & 0 \\
x_4 & x_5 & 0 & x_8 & -x_7 & 0 & 0 & 0 \\
x_6 & 0 & -x_8 & 0 & 0 & 0 & 0 & 0\\
0 & x_6 & x_7 & 0 & 0 & 0 & 0 & 0\\
x_7 & x_8 & 0 & 0 & 0 & 0 & 0 & 0\\
0 & 0 & 0 & 0 & 0 & 0 & 0 & 0\\
0 & 0 & 0 & 0 & 0 & 0 & 0 & 0
\end{pmatrix}.$$
Proceeding similarly for the Lie algebra $\mathfrak f_8$, we get the matrix
$$C_{\mathfrak f_8}=\begin{pmatrix}
0 & -x_3 & -x_4 & -(x_6+2\,x_7) & -x_8 & -x_8 & 0 & 0 \\
x_3 & 0 & -x_5 & -x_8 & -x_6 & x_7 & 0 & 0 \\
x_4 & x_5 & 0 & -x_7 & -x_8 & 0 & 0 & 0 \\
x_6+2\,x_7 & x_8 & x_7 & 0 & 0 & 0 & 0 & 0\\
x_8 & x_6 & x_8 & 0 & 0 & 0 & 0 & 0\\
x_8 & -x_7 & 0 & 0 & 0 & 0 & 0 & 0\\
0 & 0 & 0 & 0 & 0 & 0 & 0 & 0\\
0 & 0 & 0 & 0 & 0 & 0 & 0 & 0
\end{pmatrix}.$$
One can now check that the rank of $C_{\mathfrak f_6}$ is equal to 4, whereas that of $C_{\mathfrak f_8}$ is equal to 6. Consequently, $n_I(\mathfrak f_6)=4\neq 2=n_I(\mathfrak f_8)$ and these two NLAs cannot be isomorphic. 
\end{proof}

This completes the proof of Theorem~\ref{main-theorem-2}. Observe that all the NLAs we have found have rational coefficients. This guarantees the existence of a lattice $\Gamma$ in the corresponding $8$-dimensional real Lie group $G$, thus making $(\Gamma\backslash G, J)$ a complex nilmanifold with an invariant complex structure (of WnN type).

\begin{remark}
The $8$-dimensional NLAs with $b_1=2$ that admit complex structures have been recently classified in~\cite{Mill24}. Such classification can be recovered 
from Theorem~\ref{main-theorem-2} and previous results on complex structures of SnN type obtained in~\cite{LUV-SnN-2}. In fact, the condition $b_1=2$ implies that the complex structure $J$ is necessarily non-nilpotent by~\cite[Proposition~15]{CFGU-dolbeault}. Now, if $J$ is WnN then the NLA is isomorphic to $\mathfrak f_4^1$, $\mathfrak f_6$ or $\mathfrak f_8$ by Theorem~\ref{main-theorem-2}, whereas if $J$ is SnN then the NLA is isomorphic to $\mathfrak g_{12}^0$ or $\mathfrak g_{12}^1$ (see \cite[Theorem 1.1]{LUV-SnN-2}).

We also observe that the relation between the aforementioned Lie algebras and those given in~\cite{Mill24} is as follows: 
$\mathfrak f_4^1\cong\mathcal L(2,4)$, $\mathfrak f_6\cong \mathfrak n_1^+(5)$, $\mathfrak f_8\cong \mathfrak n_1^+(5)_1$, $\mathfrak g_{12}^0\cong \mathfrak b_0$ and $\mathfrak g_{12}^1\cong \mathfrak b_1$. 
\end{remark}

%

\section{Pseudo-K\"ahler and neutral Calabi-Yau structures}\label{clasi-pseudoK}

\noindent In this section we study the existence of pseudo-K\"ahler and neutral Calabi-Yau structures on complex nilmanifolds $(M=\Gamma\backslash G,J)$ endowed with non-nilpotent complex structures. 
We prove that the 8-dimensional counterexample found in~\cite{LU-PuresAppl} to a conjecture in~\cite{CFU} is unique in the class of SnN complex structures. However, we find an infinite number
of new counterexamples in the class of weakly non-nilpotent complex structures.
These provide, in addition, a new infinite family of examples of (Ricci-flat) non-flat neutral Calabi-Yau structures in real dimension 8 that do not come from a hypersymplectic structure, 
as they are not complex symplectic manifolds.
We finish the section providing a topological obstruction that any pseudo-K\"ahler nilmanifold with an invariant complex structure satisfies when its real dimension is less than or equal to 8.

\medskip
We first recall some general definitions. Let $X=(M,J)$ be a complex manifold of complex dimension~$n$. 
A \emph{pseudo-K\"ahler} structure on~$X$ 
is a pseudo-Riemannian metric~$g$ that satisfies the following two conditions:


$\bullet$  $g$ is compatible with $J$, i.e. $g(JU,JV)=g(U,V)$ for any vector fields $U,V$ on $M$; 


$\bullet$  $J$ is parallel with respect to the Levi-Civita connection $\nabla$ of $g$, i.e. $\nabla J=0$.


The latter condition is equivalent to the 2-form 
$$F(U,V)=g(U,JV)$$ 
being closed. 
Thus, any pseudo-K\"ahler manifold is, in particular, a symplectic manifold.
The signature of $g$ is of the form $(2k, 2n-2k)$, for $0\leq k\leq n$.

When the complex dimension of the manifold is even, namely, $n=2m$, there are some special classes of pseudo-K\"ahler structures:

$\bullet$ A \emph{neutral K\"ahler} structure is a pseudo-K\"ahler metric $g$ with signature $(2m,2m)$.

$\bullet$  A \emph{neutral Calabi-Yau} structure is a neutral K\"ahler structure $g$ with a nowhere vanishing form $\Phi$ of bidegree $(2m,0)$ with respect to $J$
satisfying $\nabla \Phi=0$, where $\nabla$ is the Levi-Civita connection of the metric $g$. This type of manifolds are Ricci-flat.

We observe that many neutral Calabi-Yau manifolds arise from the so-called hypersymplectic structures, introduced by Hitchin in \cite{Hit} (see also \cite{DS}). 
We do not recall the definition here, but we will use the fact that any hypersymplectic structure has a \emph{complex symplectic structure} $\Omega$. This is a closed $(2,0)$-form $\Omega$ satisfying the non-degeneration condition $\Omega^m\neq0$. 

\smallskip

From now on, we will focus on nilmanifolds $M=\Gamma\backslash G$ endowed with invariant complex structures $J$. 
Suppose that the complex nilmanifold $X=(M,J)$ of complex dimension $n$ has an invariant pseudo-K\"ahler metric $g$ with $2$-form $F$. By~\cite{Salamon}, there always exists a closed (non-zero) invariant form $\Phi$ on $M$ of bidegree $(n,0)$ with respect to $J$, so $\nabla \Phi=0$.
Therefore, any invariant pseudo-K\"ahler metric~$F$ on a complex nilmanifold~$X$ is Ricci-flat (see \cite{FPS}).

\smallskip

In \cite{CFU}, it is conjectured that the existence of a pseudo-K\"ahler metric on a complex nilmanifold $X=(M,J)$ implies the nilpotency of the complex structure $J$, in the sense of Definition~\ref{tipos_J}.
This holds true up to dimension $n=3$ \cite{CFU}, but a counterexample is given by the authors in \cite{LU-PuresAppl}. We next show that the counterexample in \cite{LU-PuresAppl} is unique in the class of strongly non-nilpotent complex structures. Nonetheless, we find an infinite family of new counterexamples in the class of weakly non-nilpotent complex structures. Indeed, the following theorem provides the classification of non-nilpotent complex structures in real dimension eight admitting (invariant or not) pseudo-K\"ahler metrics.

\begin{theorem}\label{class-pK-nN-dim8}
Let $M$ be an $8$-dimensional nilmanifold endowed with 
an invariant non-nilpotent complex structure $J$. Denote by $\mathfrak g$ the Lie algebra associated to $M$. There
exists a pseudo-K\"ahler structure on $X=(M,J)$ if and only if there is a $(1,0)$-basis $\{\omega^k\}_{k=1}^4$ for $(\frg,J)$
where the structure equations are one of the following:
\begin{itemize}
\item[i)] If $J$ is weakly non-nilpotent,
\begin{equation}\label{ecus-pk-WnN}
	\begin{gathered}[c]
	d\omega^1 = 0,\quad \
	d\omega^2 = \omega^{13}+\omega^{1\bar{3}},\quad \
	d\omega^3 = i \delta\,(\omega^{1\bar{2}} - \omega^{2\bar{1}}),
	\\
	d\omega^4 = a\,\omega^{12}+\omega^{23}+B\,\omega^{1\bar{1}}+\omega^{2\bar{3}},
	\end{gathered}
\end{equation}
where $\delta=\pm 1$ and $a=0$ with $B\in\{0,1\}$, $a=1$ with $B\in\mathbb R^{\geq 0}$, or $a=1$ with $B\in\mathbb C$ satisfying $\Imag B>0$.
Moreover, any invariant pseudo-K\"ahler metric on $X=(M,J)$ with fundamental form $F$ satisfies
\begin{equation*}
	\begin{split}
	F =&\ i\,u\,\omega^{1\bar{1}} - i\,s\,\omega^{2\bar{2}} - i\,\delta\,r\,\omega^{3\bar{3}} 
	+ v\,\omega^{1\bar{2}} - v\,\omega^{2\bar{1}} + i\,a\,\delta(r-i\,s)\,\omega^{1\bar{3}} \\[-2pt]
	&+ i\,a\,\delta(r+i\,s)\,\omega^{3\bar{1}} + (r+i\,s)\,\omega^{1\bar{4}} - (r-i\,s)\,\omega^{4\bar{1}},
	\end{split}
\end{equation*}
where $r,s,u,v\in\mathbb R$ and $rs\neq 0$.
\item[ii)] If $J$ is strongly non-nilpotent,
\begin{equation*}\label{pk-SnN}
	\begin{gathered}[c]
	d\omega^1=0,\quad \
	d\omega^2=\omega^{14}+\omega^{1\bar 4},\quad \
	d\omega^3=\omega^{12}+\omega^{1\bar 2}-\omega^{2\bar 1},
	\\
	d\omega^4=i\,\omega^{1\bar 3} - i\,\omega^{3\bar 1}.
	\end{gathered}
\end{equation*}
Moreover, the fundamental form $F$ of any invariant pseudo-K\"ahler metric on $X=(M,J)$ is
\begin{equation*}
	\begin{split}
	F = &\ i\,r\,\omega^{1\bar{1}} + i\,s\,\omega^{4\bar{4}}  
	+ u\,\omega^{1\bar{2}} - u\,\omega^{2\bar{1}}  
	+ v\, \omega^{1\bar{3}} - v\, \omega^{3\bar{1}} 
	- s\,\omega^{2\bar{3}} + s\, \omega^{3\bar{2}},
	\end{split}
\end{equation*}
where $r,s,u,v\in\mathbb R$ and $rs\neq 0$.
\end{itemize}
\end{theorem}
					
\begin{proof}
First, we recall that the existence of a pseudo-K\"ahler structure on $(M,J)$ 
is equivalent to the existence of an invariant one (see \cite[Proposition 2.1]{LU-PuresAppl}). This is due to 
Nomizu's theorem for the de Rham cohomology of nilmanifolds together with the well-known symmetrization process, which implies that any closed $k$-form
$\alpha$ on a nilmanifold is cohomologous to the invariant $k$-form $\widetilde\alpha$
obtained by symmetrization on $M$. More precisely, since $J$ is invariant, 
if there exists a pseudo-K\"ahler structure~$F$ on $(M,J)$, then $[\widetilde F]=[F]$  in the second de Rham cohomology group $H_{\rm dR}^2(M;\mathbb{R})$ and $[{\widetilde F}^3]=[F]^3 \not=0$. Hence, ${\widetilde F}$ is an invariant non-degenerate real (1,1)-form which is closed, i.e. an invariant pseudo-K\"ahler structure on $(M,J)$.
						
Let $\frg$ be the NLA underlying $M$. 
Let $F$ be any real $(1,1)$-form on $(\mathfrak g, J)$. 
If $\{\omega^k\}_{k=1}^4$ is a (1,0)-basis for $(\frg,J)$, 
then any pseudo-K\"ahler structure $F$ can be written as  
\begin{equation}\label{formaFund}
	F=\sum_{k=1}^4 i\,x_{k\bar k}\,\omega^{k\bar k}
	\ +\sum_{1\leq k<l\leq 4}\big( x_{k\bar l}\,\omega^{k\bar l}-\bar x_{k\bar l}\,\omega^{l\bar k} \big),
\end{equation}
for some coefficients $x_{k\bar k}\in\mathbb R$ and $x_{k\bar l}\in\mathbb C$. 
Note that $F$ is closed if and only if $\partial F=0$. 
We can now make use of the structure equations of non-nilpotent complex structures in dimension eight to compute $\partial F$. 
						
\smallskip
						
We start with weakly non-nilpotent complex structures, which are parametrized by~\eqref{structure-eq-WnN} in Theorem~\ref{main-theorem} 
with the tuple $(\varepsilon,\delta,\nu,a,B)$ taking the possible values specified in that theorem. 
From the condition $\partial F=0$ it is straightforward to see that $x_{2\bar3}=x_{2\bar4}=x_{3\bar4}=0$, together with 
the equalities
$$
\nu\,x_{4\bar4}=0, \qquad
a\,x_{4\bar4}=0, \qquad
B\,x_{4\bar4}=0.
$$
Hence, two cases can be distinguished depending on the vanishing of the triple $(\nu,a,B)$:
					
\smallskip
						
\noindent $\bullet$ If $(\nu,a,B)=(0,0,0)$, then $(\frg,J)$ is the product of a $6$-di\-men\-sion\-al NLA
endowed with a SnN complex structure and a complex torus. We have
$$
\partial F = -i\delta\,x_{1\bar3}\,\omega^{12\bar1}+(x_{1\bar2}-\bar x_{1\bar2}-\varepsilon\,x_{3\bar3})\,\omega^{13\bar1}
+(i\,x_{2\bar2} +\delta x_{3\bar3})\,\omega^{13\bar2} 
+(i\,x_{2\bar2}- \delta x_{3\bar3})\,\omega^{23\bar1}.
$$
If the form $F$ is closed then $x_{1\bar3}=0$ and $x_{2\bar2}=x_{3\bar3}=0$ (recall that $\delta=\pm 1$ and $x_{k\bar{k}}\in\mathbb R$), but this implies $F^4=0$. 
Thus, there are no pseudo-K\"ahler structures in this case.
						
\smallskip
						
\noindent $\bullet$ If $(\nu,a,B)\neq (0,0,0)$, then $x_{4\bar4}=0$ and one gets
$$
\begin{array}{lcl}
\partial F  \!&\!\!=\!\!&\!  -(a\,\bar x_{1\bar4} +i\delta\,x_{1\bar3})\,\omega^{12\bar1}
	+\big(x_{1\bar2}-\bar x_{1\bar2}-\varepsilon\,(x_{3\bar3} - 2\,\nu \delta\,x_{1\bar4})\big)\,\omega^{13\bar1}\\[5pt]
	&& +(i\,x_{2\bar2}+\nu\,x_{1\bar4} + \delta x_{3\bar3})\,\omega^{13\bar2} 
	+(i\,x_{2\bar2}-\nu\,\bar x_{1\bar4} - \delta\, x_{3\bar3})\,\omega^{23\bar1},\\[6pt]
F^4  \!&\!\!=\!\!&\!  -24\,x_{2\bar2}\,x_{3\bar3}\,|x_{1\bar4}|^2\,\omega^{1234\bar1\bar2\bar3\bar4}.
\end{array}
$$
The closedness of $F$ implies $x_{1\bar3}= i\,a \delta\,\bar x_{1\bar4}$, 
$x_{2\bar 2}=-\nu\,\Imag\!(x_{1\bar4})$ and $x_{3\bar3}= -\nu\,\delta\,\Real\!(x_{1\bar4})$, because $x_{2\bar2},x_{3\bar3}\in\mathbb R$.
Since $\nu=0$ gives $x_{2\bar2}=x_{3\bar3}=0$, and thus $F^4=0$, 
we next assume $\nu=1$ and choose $\Real\!(x_{1\bar4}), \Imag\!(x_{1\bar 4})\neq 0$ in order to 
have $x_{2\bar2},x_{3\bar3}\neq 0$.
Furthermore, to ensure the condition $\partial F=0$ one also needs
$$\begin{array}{lcl}
0 \!&\!\!=\!\!&\!  2\,i\,\Imag\!(x_{1\bar2})-\varepsilon\,(x_{3\bar3} - 2 \delta\,x_{1\bar4})\\[4pt]
\!&\!\!=\!\!&\!  2\,i\,\big(\Imag\!(x_{1\bar2}) + \varepsilon\,\delta\,\Imag\!(x_{1\bar4})\big)
+ 3\,\varepsilon\,\delta\,\Real\!(x_{1\bar4}).
\end{array}$$
Since we have chosen $\Real\!(x_{1\bar4})\neq 0$, 
the only possible pseudo-K\"ahler structures arise when $\varepsilon=0$. 
With this choice, it suffices to take $\Imag\!(x_{1\bar2})=0$ to have $\partial F=0$. This proves part~\textrm{i)} of the statement of the theorem. We simply observe that the values of $a$ and $B$ come directly from Theorem~\ref{main-theorem}, due to the conditions $\varepsilon=0$ and $\nu=1$. 
						
\medskip
						
It remains to study those complex structures of SnN type, which are classified into two families 
in \cite[Theorem 3.3]{LUV-SnN-2}.
The first family is given by
$$\text{Family I:} \quad
\begin{cases}
	d\omega^1 = 0,\\
	d\omega^2 = \varepsilon\,\omega^{1\bar 1},\\
	d\omega^3 = \omega^{14}+\omega^{1\bar 4}+a\,\omega^{2\bar 1}
			+ i\,\delta\,\varepsilon\,b\,\omega^{1\bar 2},\\
	d\omega^4 = i\,\nu\,\omega^{1\bar 1} +b\,\omega^{2\bar 2}+ i\,\delta\,(\omega^{1\bar 3}-\omega^{3\bar 1}),
\end{cases}
$$
where $\delta=\pm 1$, $\varepsilon,\nu\in\{0,1\}$, and $(a,b)\in \mathbb R^2-\{(0,0)\}$ with $a\geq 0$.
						
A direct calculation
shows that $\partial F=0$ implies $x_{1\bar4}=x_{2\bar3}=x_{2\bar4}=x_{3\bar4}=0$. 
With this choice, we obtain:
$$
\begin{array}{lcl}
\partial F \!&\!\!=\!\!&\!   -i\,\varepsilon\,(x_{2\bar2} + b\, \delta\,x_{1\bar3})\,\omega^{12\bar1} -i\,a\,x_{3\bar3}\,\omega^{13\bar2}
	+(x_{1\bar3}-\bar x_{1\bar3} - \nu\,x_{4\bar4})\,\omega^{14\bar1} \\[4pt]
	&& +(i\,x_{3\bar3} + \delta\, x_{4\bar4})\,\omega^{14\bar3}  -\varepsilon\,b\,\delta\,x_{3\bar3}\,\omega^{23\bar1} - i\,b\,x_{4\bar4}\,\omega^{24\bar2}  +(i\,x_{3\bar3} -\delta\, x_{4\bar4})\,\omega^{34\bar1},\\[6pt]
F^4 \!&\!\!=\!\!&\!  24\,x_{4\bar4}\,(x_{1\bar1}\,x_{2\bar2}\,x_{3\bar3}-x_{3\bar3}\,|x_{1\bar2}|^2-x_{2\bar2}|x_{1\bar3}|^2)\,
\omega^{1234\bar1\bar2\bar3\bar4}.
\end{array}
$$
Since $x_{3\bar3}$ and $x_{4\bar4}$ are real numbers, one has that $i\,x_{3\bar3} +\delta\, x_{4\bar4}=0$ if and
only if $x_{3\bar3}=x_{4\bar4}=0$. However, this yields $F^4=0$, contradicting the non-degeneracy condition.
Therefore, there are no pseudo-K\"ahler structures for SnN complex structures in Family I. 
						
The SnN complex structures in the second family are parametrized by 
$$\text{Family II:} \quad
	\begin{cases}
	d\omega^1=0,\\
	d\omega^2=\omega^{14}+\omega^{1\bar 4},\\
	d\omega^3=a\,\omega^{1\bar 1}
		+\varepsilon\,(\omega^{12}+\omega^{1\bar 2}-\omega^{2\bar 1})
		+i\,\mu\,(\omega^{24}+\omega^{2\bar 4}),\\
	d\omega^4=i\,\nu\,\omega^{1\bar 1}-\mu\,\omega^{2\bar 2}+i\,b\,(\omega^{1\bar 2}-\omega^{2\bar 1})+i\,(\omega^{1\bar 3}-\omega^{3\bar 1}),
	\end{cases}$$
where $a, b\in\mathbb R$, and $\varepsilon,\mu,\nu\in\{ 0,1\}$ with $(\varepsilon,\mu)\neq (0,0)$ and $\mu\nu=0$.
						
Calculating the condition
$\partial F=0$ from the previous structure equations, we directly get $x_{2\bar4}=x_{3\bar4}=0$, together with the equalities $\varepsilon\,x_{3\bar3}=0$ and $\mu\,x_{3\bar3}=0$. Since $(\varepsilon,\mu)\neq (0,0)$ we have $x_{3\bar3}=0$. Now, the conditions $x_{2\bar4}=x_{3\bar3}=x_{3\bar4}=0$ imply $x_{1\bar4}=0$
and $x_{2\bar3}=-x_{4\bar4}$. 
In this way, the $(2,1)$-form $\partial F$ reduces to
$$
\begin{array}{lcl}
\partial F \!&\!\!=\!\!&\!  \big(a\,x_{4\bar4}+\varepsilon\,(x_{1\bar3}-\bar x_{1\bar3})\big)\,\omega^{12\bar1} 
	-\big(\nu\,x_{4\bar4} - (x_{1\bar2}-\bar x_{1\bar2})\big)\,\omega^{14\bar1}\\[4pt]
	&& +i\,(x_{2\bar2}-i\,b\,x_{4\bar4}-\mu\,x_{1\bar3})\,\omega^{14\bar2}
	+i\,(x_{2\bar2}+i\,b\,x_{4\bar4}-\mu\,\bar x_{1\bar3})\,\omega^{24\bar1} 
	+3\,i\,\mu\,x_{4\bar4}\,\omega^{24\bar2},\\[6pt]
F^4 \!&\!\!=\!\!&\!  -24\,x_{4\bar4}\,\Big(x_{1\bar1}\,x_{4\bar4}^2+x_{2\bar2}\,|x_{1\bar3}|^2
	-i\,x_{4\bar4}\,(x_{1\bar2}\bar x_{1\bar3}-\bar x_{1\bar2}\,x_{1\bar3})\Big)\,
	\omega^{1234\bar1\bar2\bar3\bar4}.
\end{array}
$$
It is clear that if $\mu=1$, then $x_{4\bar4}=0$ and the non-degeneracy condition fails. Hence, $\mu=0$, which in addition implies $\varepsilon=1$. 
Observe that now $\partial F=0$ holds if and only if 
$$
a\,x_{4\bar4} + 2\, i\,\Imag\!(x_{1\bar3})  =0, \quad\ 
\nu\,x_{4\bar4} - 2\, i\,\Imag\!(x_{1\bar2}) =0, \quad\ 
x_{2\bar2}+i\,b\,x_{4\bar4}=0.
$$
Since $x_{2\bar2},x_{4\bar4}$ are real numbers, with $x_{4\bar4} \neq0$, we get that 
$a=b=\nu=0$ and $\Imag\!(x_{1\bar2})=\Imag\!(x_{1\bar3})=x_{2\bar2}=0$. In this case, the form $F^4$ is given by 
$$
F^4 = -24\,x_{1\bar1}\,x_{4\bar4}^3\, \omega^{1234\bar1\bar2\bar3\bar4},
$$
so it suffices to impose the extra condition $x_{1\bar1}\neq0$ to get pseudo-K\"ahler structures. This leads to part \textrm{ii)} of the theorem.
						
This finishes the study of pseudo-K\"ahler geometry on those $8$-dimensional NLAs endowed with non-nilpotent complex structures, thus attaining the desired result.
\end{proof}

In the following, we prove that all the invariant pseudo-K\"ahler metrics obtained in Theorem~\ref{class-pK-nN-dim8} are non-flat (recall that they always are Ricci-flat).
					
\begin{proposition}\label{non-flat-pK}
Let $X=(M,J)$ be a pseudo-K\"ahler nilmanifold with $\dim_{\mathbb C}X=4$ and 
an invariant non-nilpotent $J$. Every invariant pseudo-K\"ahler metric on $X$ is non-flat. 
\end{proposition}
					
\begin{proof}
Recall that the structure equations of each $X=(M,J)$ 
appear in Theorem~\ref{class-pK-nN-dim8}.
If the complex structure~$J$ is SnN, the result can be found in \cite[Proposition 3.4]{LU-PuresAppl}. Thus, we need to prove the statement for WnN complex structures, namely, those given by case~\textrm{i)} in 
Theorem~\ref{class-pK-nN-dim8}.
						
We first note that the pseudo-Riemannian metric $g$ can be recovered from its fundamental form $F$ by $g(Z,W)=F(JZ,W)$, where $Z,W$ are any complex vector fields on the complex nilmanifold $X$. 
Since the pseudo-K\"ahler structures under consideration are invariant, the (complexified) Koszul formula for the Levi-Civita connection $\nabla$
of the metric $g$ reduces to
$$
2 g(\nabla_UV,W)=g([U,V],W)-g([V,W],U)+g([W,U],V),
$$
for invariant complex vector fields $U,V,W$ on $X$.
						
Now, since the structure $(J,g)$ is pseudo-K\"ahler, we have $\nabla J=0$, which is equivalent to the condition $\nabla_U(JV)-J(\nabla_UV)=0$, for any complex vector fields $U,V$ on $X$. 
Therefore, if $V$ has type (1,0) then $J(\nabla_UV) = \nabla_U(JV) =i\, (\nabla_UV)$, that is to say,
$\nabla_UV$ is also a vector field of bidegree (1,0). In particular, if $\{Z_j\}_{j=1}^4$ denotes the basis of complex vector fields of bidegree (1,0) dual to the
basis $\{\omega^j\}_{j=1}^4$, one has that $\nabla_U Z_j$ has type $(1,0)$ for every $1 \leq j \leq 4$.
Moreover, by complex conjugation, it suffices to compute
$\nabla_{Z_k}Z_j$ and $\nabla_{\overline{Z}_k}Z_j$ for
$1 \leq j,k \leq 4$.

To prove the result we need the following derivations, which can be obtained by direct calculation from Theorem~\ref{class-pK-nN-dim8}~{\rm i)}:
\begin{equation}\label{nabla-1}
\nabla_{\bar{Z}_1}Z_2 = -i\delta\, Z_3, \quad\quad 
\nabla_{Z_1}Z_3 = -\frac{ir}{s}\, Z_2 +\frac{irv}{s(r-is)}\, Z_4, 
\end{equation}
and
\begin{equation}\label{nabla-2}
	\nabla_{Z_1}Z_2 = 0,\quad\quad 
	\nabla_{Z_4}Z_2 = 0,\quad\quad
	\nabla_{\bar{Z}_4}Z_2 = 0.
\end{equation}
						
Let $R$ be the curvature tensor of the pseudo-K\"ahler metric, i.e. 
$$
R(U,V,W,T)=g\big(\nabla_U\nabla_VW-\nabla_V\nabla_UW-\nabla_{[U,V]}W, \, T\big),
$$
for $U,V,W,T$ complex vector fields on $X$. 
Taking into account the observations above, by complex conjugation and using the symmetries of the curvature tensor, one has that 
the metric $g$ is non-flat if and only if
$R(Z_i,\bar{Z}_j,Z_k,\bar{Z}_l)\not=0$ for some $i,j,k,l$.
We next show that $R(Z_1,\bar{Z}_1,Z_2,\bar{Z}_2)$ is not zero.
						
From the complex equations~\eqref{ecus-pk-WnN}, we get $[Z_1,\bar{Z}_1]=-B\,Z_4+\bar{B}\,\bar{Z}_4$.  
Therefore, from \eqref{nabla-1}-\eqref{nabla-2} it follows that
$$
\begin{array}{rl}
R(Z_1,\bar{Z}_1,Z_2,\bar{Z}_2)
	\!\!&\!\!= \, g(\nabla_{Z_1}\nabla_{\bar{Z}_1}Z_2 - \nabla_{\bar{Z}_1}\nabla_{Z_1}Z_2 - \nabla_{[Z_1,\bar{Z}_1]}Z_2, \, \bar{Z}_2) \\[5pt]
	&= g(\nabla_{Z_1}\nabla_{\bar{Z}_1}Z_2 
	- \nabla_{\bar{Z}_1}\nabla_{Z_1}Z_2 
	+ B\,\nabla_{Z_4}Z_2 - \bar{B}\,\nabla_{\bar{Z}_4}Z_2, 
	\, \bar{Z}_2) \\[5pt]
	&=\, -i \delta\, g\!\left(\nabla_{Z_1} Z_3, \, \bar{Z}_2\right) \\[5pt]
	& =-i \delta\big(\!-\frac{ir}{s}\, g(Z_2,\bar{Z}_2) +\frac{irv}{s(r-is)}\, g(Z_4,\bar{Z}_2)  \big)\\[5pt]
	&= - \delta\, r \not= 0,
\end{array}
$$
where the last equality comes from the fact that $g(Z_2,\bar{Z}_2)= i\, F(Z_2,\bar{Z}_2)=s$ and $g(Z_4,\bar{Z}_2)= i\, F(Z_4,\bar{Z}_2)=0$, according to the description of $F$ in Theorem~\ref{class-pK-nN-dim8}~\textrm{i)}.
\end{proof}

In the following result we study the existence of neutral Calabi-Yau metrics 
on 8-dimensional nilmanifolds with non-nilpotent complex structures. Recall that a neutral Calabi-Yau structure is, in particular, pseudo-K\"ahler.
					
\begin{proposition}\label{neutralCY-nil}
All the complex nilmanifolds $X=(M,J)$ given in Theorem~\ref{class-pK-nN-dim8} admit neutral Calabi-Yau structures.
\end{proposition}
					
\begin{proof}
Let us show that neutral Calabi-Yau structures exist for any non-nilpotent complex structure in Theorem~\ref{class-pK-nN-dim8}. Since the result for those complex structures in part~{\rm ii)} of the theorem was proved in \cite[Proposition 3.4]{LU-PuresAppl}, it suffices to study the case~{\rm i)}. 
In particular, we need to prove that every weakly non-nilpotent complex structure given by~\eqref{ecus-pk-WnN} has a pseudo-K\"ahler metric of neutral signature $(4,4)$. 
						
Let us consider the basis of real $1$-forms 
$\{e^1,\ldots,e^8\}$ defined by $e^{2k-1}+i\, e^{2k}=\omega^k$, for $1 \leq k \leq 4$, where $\{\omega^1,\ldots,\omega^4\}$ is the basis of $(1,0)$-forms in Theorem~\ref{class-pK-nN-dim8} 
satisfying the complex structure equations~\eqref{ecus-pk-WnN}.
In terms of this real basis $\{e^1,\ldots,e^8\}$, the corresponding complex structure $J$ and pseudo-K\"ahler metrics $F$ are expressed as
$$
\begin{array}{rl}
	& J e^1 = -e^2, \quad J e^3 = -e^4, \quad J e^5 = -e^6, \quad J e^7 = -e^8, \\[8pt]
	& F = \, 2u\, e^{12} + 2v\, e^{13} + 2a\,\delta\,s\, e^{15} +2a\,\delta\,r\, e^{16} + 2r\, e^{17} + 2s\, e^{18} + 2v\, e^{24} - 2a\,\delta\,r\, e^{25} \\[4pt]
	& \phantom{F = \, } + 2a\,\delta\,s\, e^{26} - 2s\, e^{27} + 2r\, e^{28} - 2s\, e^{34} - 2\delta\,r\, e^{56},
\end{array}
$$
where $r,s,u,v\in\mathbb R$ with $rs\neq 0$. By direct calculation, the pseudo-Riemannian metric $g(x,y)=F(Jx,y)$ is 
given in terms of the real dual basis by the following matrix:
\begin{equation}\label{neutral-g}
	(g_{ij})_{i,j}= \left(
		\begin{array}{cccccccc}
		2u & 0  & 0  & -2v & 2a\,\delta\,r  & -2a\,\delta\,s & 2s & -2r \\[2pt]
		0  & 2u & 2v & 0  & 2a\,\delta\,s & 2a\,\delta\,r  & 2r & 2s \\[2pt]
		0  & 2v & -2s  & 0  & 0  & 0  & 0 & 0 \\[2pt]
		-2v & 0 & 0  & -2s  & 0 & 0  & 0 & 0 \\[2pt]
		2a\,\delta\,r  & 2a\,\delta\,s & 0 & 0  & -2\delta\,r\  & 0   & 0 & 0 \\[2pt]
		-2a\,\delta\,s & 2a\,\delta\,r & 0 & 0 & 0  & -2\delta\,r   & 0 & 0 \\[2pt]
		2s & 2r & 0 & 0 & 0 & 0 & 0 & 0 \\[2pt]
		-2r & 2s & 0 & 0 & 0 & 0 & 0 & 0
	\end{array}
\right).
\end{equation}
It is easy to see that there are metrics in the family~\eqref{neutral-g} with neutral signature for any value of the triple $(\delta,a,B)$ in the complex equations~\eqref{ecus-pk-WnN}. For instance:
							
\noindent 
$\bullet$ If $a=0$ and $B\in\{0,1\}$, take $u=v=0$. Then, the eigenvalues of the matrix~\eqref{neutral-g} are 
$$-2\,s,\quad -2\sqrt{r^2+s^2},\quad 2\sqrt{r^2+s^2},\quad -2\,\delta\,r,$$
all with algebraic multiplicity two. It suffices to choose $\delta\,r\,s<0$ to find pseudo-K\"ahler metrics with signature $(4,4)$.
							
\noindent 
$\bullet$ If $a=1$ and either $B\in\mathbb R^{\geq 0}$ or $B\in\mathbb C$ with $\Imag B>0$,
consider $u=1$, $v=0$. The leading principal minors of~\eqref{neutral-g}, written in increasing order, are
$$2,\quad 4,\quad -8\,s,\quad 16\,s^2,\quad -32\,s^2\,(\delta\,r+r^2+s^2),\quad 
64s^2\big((r^2 + s^2)\delta + r\big)^2,$$
$$-128\,\delta\,r\,s^2\,(r^2+s^2)\,(\delta\,r+r^2+s^2),\quad 256\,r^2\,s^2\,(r^2+s^2)^2.$$
If we take $s<0$ and $\delta\,r>0$, one can check there are exactly four changes of sign in the previous sequence of minors. Therefore, the matrix~\eqref{neutral-g} has exactly four negative eigenvalues and four positive ones, thus the signature of this metric is $(4,4)$.
						
\smallskip

Finally, since the $(4,0)$-form $\Phi=\omega^{1234}$ is parallel,
we conclude that $X$ has neutral Calabi-Yau structures.
\end{proof}

An interesting consequence of the previous results is that 
the families of pseudo-K\"ahler metrics obtained in Theorem~\ref{class-pK-nN-dim8} 
provide new neutral Calabi-Yau metrics in eight dimensions that are not flat (although they all are 
Ricci-flat).
					
\begin{theorem}\label{new-neutral-CY}
Let $X=(M,J)$ be a pseudo-K\"ahler nilmanifold with $\dim_{\mathbb C}X=4$ and an invariant non-nilpotent $J$. Then, $X$ has neutral Calabi-Yau metrics that are Ricci-flat but not flat. 
\end{theorem}
					
\begin{proof}
It is a direct consequence of Propositions~\ref{non-flat-pK} and~\ref{neutralCY-nil}. 
\end{proof}

Several constructions of hypersymplectic structures on nilpotent Lie algebras can be found in recent literature. Examples of 2-step nilpotent Lie algebras (of Kodaira type) are given in~\cite{FPPS} and examples in the 3-step nilpotent case are obtained in~\cite{AD}. The first 4-step nilpotent Lie algebras with hypersymplectic structures are found in~\cite{BGL}. Other examples have been constructed  in~\cite{ContiGil} using semidirect products of Lie algebras. All these hypersymplectic nilpotent Lie algebras are, in particular, neutral Calabi-Yau (hence, also pseudo-K\"ahler) and it is worthy to remark that their underlying complex structures $J$ are always of nilpotent type. 
					
We next prove that the neutral Calabi-Yau structures constructed in Theorem~\ref{new-neutral-CY} do not come from hypersymplectic structures. Moreover, in the following result we show that non-nilpotent complex structures in real dimension 8 do not admit any complex symplectic form, thus extending the non-existence result given in \cite[Proposition 5.8]{BFLM} for SnN complex structures. 
					
\begin{proposition}\label{non-existence-hol-symplectic}
Let $X=(M,J)$ be a complex nilmanifold with $\dim_{\mathbb C}X=4$ and a non-nilpotent complex structure $J$. Then, $X$ does not admit any complex symplectic structure. In particular, the non-flat neutral Calabi-Yau metrics constructed in Theorem~\ref{new-neutral-CY} do not arise from any hypersymplectic structure.
\end{proposition}
					
\begin{proof}
Similarly to the argument given in the proof of Theorem~\ref{class-pK-nN-dim8}, the existence of a complex symplectic structure $\Omega$ on $X=(M,J)$ is equivalent to the existence of an invariant one (see~\cite[Proposition~6.1]{BFLM}). In fact, since $J$ is invariant, we get that $\widetilde  \Omega$ is an invariant closed (2,0)-form on $X$ such that $[\widetilde \Omega]=[\Omega]$ in the second complexified de Rham cohomology group $H_{\rm dR}^2(X;\mathbb{C})$. Since $X$ is compact, $[\Omega]^2 \not=0$ by Stokes theorem. Hence, $[{\widetilde \Omega}^2]=[\Omega]^2 \not=0$, so ${\widetilde \Omega}$ is non-degenerate and it is an invariant complex symplectic structure on $X$.
						
Therefore, it suffices to show the non-existence result for invariant complex symplectic structures.  
This result is proved in \cite[Proposition 5.8]{BFLM} for SnN complex structures.
Hence, we here focus on WnN complex structures. Recall that these are parametrized by \eqref{structure-eq-WnN} in Theorem~\ref{main-theorem} 
with the tuple $(\varepsilon,\delta,\nu,a,B)$ taking the possible values specified in that theorem. 
						
Consider any invariant (2,0)-form $\Omega$, i.e.,
$$
\Omega = \alpha\,\omega^{12} + \beta\,\omega^{13} + \gamma\,\omega^{14} + \tau\,\omega^{23} + \theta\,\omega^{24} + \xi\,\omega^{34},
$$
where $\alpha,\beta,\gamma,\tau,\theta,\xi \in {\mathbb C}$. 
Since $\Omega^2= 2(\alpha\xi-\beta\theta+\gamma\tau)\omega^{1234} $, we have that $\Omega^2\not=0$ if and only if $\alpha\xi-\beta\theta+\gamma\tau\neq0$. Note that $d\Omega=0$ if and only if $\partial\Omega=0$ and $\bar{\partial}\Omega=0$.
						
By \eqref{structure-eq-WnN}, the (3,0)-form $\partial\Omega$ is given by
$$
\partial\Omega=-(\gamma\nu+a\xi)\omega^{123} + \theta \omega^{134},
$$
so $\partial\Omega=0$ is equivalent to $\gamma\nu+a\xi=0$ and $\theta=0$. 
Taking into account these conditions, we arrive at 
$$
\begin{array}{rl}
	& \overline\partial\,\Omega=\, i(\beta\delta+\tau\varepsilon)\omega^{12\bar{1}}
	+i\tau\delta\,\omega^{12\bar{2}}
	-\gamma\nu\,\omega^{12\bar{3}}
	+\xi B\,\omega^{13\bar{1}}
	-(\tau-2\xi\delta\varepsilon\nu)\omega^{13\bar{3}} \\[4pt]
	& \phantom{\overline\partial\,\Omega=} -i\xi\varepsilon\,\omega^{14\bar{1}}
	-i\xi\delta\,\omega^{14\bar{2}}
	+\xi\nu\,\omega^{23\bar{3}}
	+i\xi\delta\,\omega^{24\bar{1}}.
\end{array}
$$
Therefore, if the (2,0)-form $\Omega$ is closed then $\tau=0=\xi$, and one has $\Omega^2=0$, i.e. the form is degenerate. Thus, there are no complex symplectic structures when $J$ is non-nilpotent.
						
The statement about neutral Calabi-Yau structures comes from the fact that any hypersymplectic structure is, in particular, complex symplectic.
\end{proof}

In view of Proposition~\ref{non-existence-hol-symplectic}, the following general question arises: \emph{Does the existence of a complex symplectic structure imply the nilpotency of the complex structure $J$?}
										
\vskip.3cm
					
It is well-known that the existence of a (positive definite) K\"ahler metric on a compact manifold imposes strong restrictions on the topology of the manifold. In the pseudo-K\"ahler case, as far as we know, no topological obstruction is known on the compact complex manifold, apart from those conditions coming from the existence of a symplectic structure on the underlying real manifold. The following results suggest a possible restriction on the first Betti number of pseudo-K\"ahler manifolds, at least for 
nilmanifolds with invariant complex structures. Another restriction seems to be provided on the step of nilpotency of the pseudo-K\"ahler nilmanifolds.

\begin{proposition}\label{Betti1}
Let $M$ be an $8$-dimensional nilmanifold endowed with a non-nilpotent complex structure $J$.
If $(M,J)$ admits a pseudo-K\"ahler structure, then the NLA associated to $M$ is isomorphic to one of the following:
\begin{equation*}
	\begin{split}
		& \mathfrak f_5^\gamma = (0,\,0,\,0,\,13,\,23,\,34,\,\frac{\gamma}{2}\cdot 12+35,\,14+25),\\	
		& \mathfrak f_7 = (0,\,0,\,0,\,13,\,23,\,14+25,\, 2\cdot 14+34,\, 15+24+35), \\
		& \mathfrak g^0_{10}=(0,\,0,\,0,\,13,\,23,\,14+25,\,15+24,\,16+27),
	\end{split}
\end{equation*}
with $\gamma\in\{0,1\}$.
In particular, the nilmanifold $M$ is either $3$-step or $4$-step nilpotent, and its first Betti number is $b_{1}(M)=3$.
\end{proposition}
					
\begin{proof}
From i) in Theorem~\ref{class-pK-nN-dim8} and Table~\ref{tab:cambios-base-real}, one can deduce that the NLAs associated to the $8$-dimensional nilmanifolds with a WnN complex structure that admit pseudo-K\"ahler metrics 
are $\mathfrak f_5^0$, $\mathfrak f_5^1$ and $\mathfrak f_7$. 
Moreover, every WnN complex structure on these NLAs gives rise to a pseudo-K\"ahler structure. The Lie algebra $\mathfrak g^0_{10}$ is associated to the SnN complex structure in Theorem~\ref{class-pK-nN-dim8} and it comes from~\cite{LUV-SnN-2}.
We observe that this NLA has another SnN complex structure (see~\cite[Table~5]{LUV-SnN-2}) that does not admit pseudo-K\"ahler metrics.
					
By Nomizu's Theorem, the de Rham cohomology $H^*_{\mathrm{dR}}(M;\mathbb R)$ of the nilmanifold $M$ can be computed
at the level of its underlying Lie algebra; in particular $b_{1}(M)=b_{1}(\frg)=3$ for 
$\frg=\mathfrak f_5^\gamma,\mathfrak f_7, \mathfrak g^0_{10}$. 						
Finally, note that the step of nilpotency of the Lie algebras $\mathfrak f_5^\gamma$ and 
$\mathfrak f_7$ is $s=3$, whereas for  $\mathfrak g^0_{10}$ we have $s=4$ (see the ascending types of these NLAs in Table~\ref{tab:real-invariants} and \cite[Table 5]{LUV-SnN-2}).
\end{proof}

\begin{remark}
There exist $8$-dimensional nilmanifolds $M$ endowed with non-nilpotent complex structures that are $5$-step or whose first Betti number is $b_{1}(M)=2$. Note also that the previous result provides a necessary but not sufficient condition: the nilmanifolds with a WnN complex structure whose underlying NLA is $\mathfrak f^0_4$ are $3$-step and have first Betti number equal to $3$, but they are not pseudo-K\"ahler.
\end{remark}

In the next result we show that similar restrictions to those provided in Proposition~\ref{Betti1} can be found for a larger class of complex nilmanifolds.
					
\begin{theorem}\label{restrictions-pK}
Let $M$ be a $2n$-dimensional nilmanifold with $2\leq n\leq 4$ endowed with an invariant complex structure $J$.
If $(M,J)$ admits a pseudo-K\"ahler structure, then the nilmanifold is $s$-step with $s\leq n$, and its first Betti number satisfies $b_{1}(M)\geq 3$.
\end{theorem}
					
\begin{proof}
Recall that every invariant complex structure on a $4$-dimensional nilmanifold is nilpotent. In six dimensions, it is proved in \cite{CFU} that pseudo-K\"ahler structures only exist when $J$ is nilpotent. 
Moreover, for any nilpotent complex structure we know by 
\cite[Proposition 10 (iii)]{CFGU-dolbeault} that $s\leq n$, and one also has $b_{1}(M)\geq 3$
by \cite[Proposition 15]{CFGU-dolbeault}.
Hence, it remains to study the case when the complex structure $J$ of an 8-dimensional pseudo-K\"ahler nilmanifold is non-nilpotent. This is done in Proposition~\ref{Betti1}, so the proof is complete.
\end{proof}

\medskip

\section*{Acknowledgments}
\noindent 
We are grateful to D. Millionshchikov for pointing out the reference~\cite{Mill24}. 
We are also grateful to the anonymous referees for very useful observations and suggestions that helped us to improve the presentation of the paper. Indeed, one of the referees detected that two of the Lie algebras given in the preliminary version of this paper were isomorphic. This allowed us to correct the classification in this new version of the paper. 
This work was partially supported by grant PID2020-115652GB-I00, funded by MCIN/AEI/10.13039/ 501100011033,
and by grant E22-17R ``Algebra y Geometr\'{\i}a'' (Gobierno de Arag\'on/FEDER).

\end{document}